\documentclass[12pt,a4paper,reqno]{amsart}
\usepackage[dvips]{graphicx}
\usepackage{verbatim}
\usepackage{hyperref}
\usepackage{color}

\usepackage{amssymb}
\usepackage{latexsym}
\usepackage{amsmath}
\usepackage{bbm} 
\usepackage{amsthm}

\renewcommand{\leq}{\leqslant}
\renewcommand{\geq}{\geqslant}
\newcommand{\pvector}[1]{% column vector
  \begin{pmatrix}
    #1
  \end{pmatrix}} %
\newcommand{\ddirac}[1]{% Dirac's delta
  \,\boldsymbol{\delta}\!\pvector{#1}\!} %
\renewcommand{\d}{\,{\rm d}} % differential
\newcommand{\supp}{\operatorname{supp}}

\providecommand{\norma}[1]{\Vert #1 \Vert}

\providecommand{\ab}[1]{\vert #1\vert}

\newcommand{\R}{\mathbbm{R}} 
\newcommand{\N}{\mathbbm{N}}

\newcommand{\la}{\lambda}
\newcommand{\eps}{\varepsilon}
\newcommand{\te}{\theta}
\newcommand{\vphi}{\varphi}
\newcommand{\one}{\mathbbm 1}

\theoremstyle{plain}
\newtheorem{teo}{Theorem}[section]
\newtheorem{cor}[teo]{Corollary}
\newtheorem{lemma}[teo]{Lemma}
\newtheorem{prop}[teo]{Proposition}
\theoremstyle{definition}
\newtheorem{define}[teo]{Definition}

\newtheorem{remark}[teo]{Remark}

\numberwithin{equation}{section}

\begin{document}
\title[Extremizers for higher order Schr\"odinger equations]
{On extremizers for Strichartz estimates for higher order Schr\"odinger equations}

\author{Diogo Oliveira e Silva}
\address{
        Diogo Oliveira e Silva\\
        Hausdorff Center for Mathematics\\
        53115 Bonn, Germany}
\email{dosilva@math.uni-bonn.de}
\thanks{\it The first author was partially supported by the Hausdorff Center for Mathematics.}

 \author{Ren\'e Quilodr\'an}
\address{Ren\'e Quilodr\'an\\
        Departamento de Ciencias Exactas\\
        Universidad de Los Lagos\\
        Avenida Fuchslocher 1305, Osorno, Chile}
\email{rene.quilodran@ulagos.cl}

\subjclass[2010]{42B10}
\keywords{Fourier extension theory, extremizers, optimal constants, convolution of singular measures, concentration-compactness, Strichartz inequalities.}
\begin{abstract}
 For an appropriate class of convex functions $\phi$, we study the  Fourier extension operator on the surface 
 \mbox{$\{(y,\ab{y}^2+\phi(y)):y\in\R^2\}$}
equipped with  projection measure. 
For the corresponding extension inequality, we compute optimal constants
and prove that extremizers do not exist. The main tool is a new comparison principle for convolutions of certain singular measures that 
holds in all dimensions. Using tools of concentration-compactness flavor, we further investigate the behavior of general extremizing sequences.
Our work is directly related to the study of extremizers and optimal constants for Strichartz estimates of certain 
higher order Schr\"odinger equations.
In particular, we resolve a dichotomy from the recent literature concerning the existence of extremizers for a family of fourth order Schr\"odinger equations, and compute the corresponding operator norms exactly where only lower bounds were previously known.
\end{abstract}
\maketitle
\tableofcontents

\section{Introduction}

Recently there has been considerable interest in the study of extremizers, optimal constants, and sharp instances of various Fourier extension inequalities. 
The purpose of the present paper is three-fold.
Firstly, we establish a sharp Fourier extension inequality on certain non-compact hypersurfaces in Euclidean space.
Secondly, we use concentration-compactness tools to study the qualitative behavior of extremizing sequences for this sharp inequality.
Thirdly, we explore the link between Fourier extension inequalities and Strichartz estimates for certain higher order Schr\"odinger equations, and resolve some dichotomies concerning the existence of extremizers that have appeared in the recent literature.\\

\noindent Throughout the paper, we normalize the Fourier transform as follows:
$$\widehat{f}(y)=\int_{\R^d} f(x) e^{-i\langle x,y\rangle} \d x,$$
where $\langle\cdot,\cdot\rangle$ denotes the usual inner product in $\R^d$.
Given a sufficiently nice function $\phi:\R^d\to\R$, consider the hypersurface in $\R^{d+1}$
\begin{equation}\label{DefS}
\Sigma_\phi=\{(y,\ab{y}^2+\phi(y)):y\in\R^d\}
\end{equation}
 endowed with projection measure 
\begin{equation}\label{defprojmeas}
\sigma(y,t)=\ddirac{t-\ab{y}^2-\phi(y)}\d y\d t,
\end{equation}
which in turn is defined by requiring that the identity
\[\int_{\R^{d+1}} g(y,t)\d\sigma(y,t)=\int_{\R^d}g(y,\ab{y}^2+\phi(y))\d y\]
holds for every Schwartz function $g$.
The  Fourier extension operator for the hypersurface $\Sigma_\phi$ is defined as 
\[\widehat{f\sigma}(x,t)=\int_{\R^{d}} f(y) e^{-i\langle x, y\rangle}e^{{-}it(\ab{y}^2+\phi(y))}\d y,\;\;\;(x,t)\in\R^{d+1}.\]
Estimates for this operator  stem from the seminal works of Tomas \cite{T}, Stein \cite{S} and Strichartz \cite{St}. In particular, under certain fairly general convexity assumptions on $\phi$, the inequality
$$\norma{\widehat{f\sigma}}_{L^{2+\frac 4d}(\R^{d+1})}
\lesssim_{d,\phi} \norma{f}_{L^2(\R^d)}$$
holds in dimensions $d\geq 1$,
see e.g. \cite{JPS,JSS,KPV}.
To pursue this point further, let us specialize the discussion to the two-dimensional case $d=2$. 
Using the fact that in this case the exponent $2+\frac4 d=4$ is an
even integer together with Plancherel's Theorem, the inequality
$$\norma{\widehat{f\sigma}}_{L^4(\R^3)}\lesssim_\phi \norma{f}_{L^2(\R^2)}$$
 can be rewritten in bilinear convolution form as
\begin{equation}
 \label{nonsharp-2-convolution}
 \norma{f\sigma\ast f\sigma}_{L^2(\R^3)}\lesssim_{\phi} \norma{f}_{L^2(\R^2)}^2.
\end{equation}
We emphasize that, since the surface $\Sigma_\phi$ is not compact, inequality \eqref{nonsharp-2-convolution} does not hold in general if the projection measure is replaced by the usual surface measure on $\Sigma_\phi$. Inequality \eqref{nonsharp-2-convolution} will be established under some mild assumptions on $\phi$ in Theorem \ref{main} below. As we will see, it will follow from the fact that the convolution $\sigma\ast\sigma$ defines a measure which is absolutely continuous with respect to Lebesgue measure on $\R^3$, and whose Radon--Nikodym derivative is given by an essentially bounded function. \\

\noindent In the first part of the  paper, we address the question of existence of extremizers for the sharp version of inequality \eqref{nonsharp-2-convolution}, and compute the optimal constant. 
More precisely, consider the sharp inequality 
\begin{equation}\label{2-convolution}
 \norma{f\sigma\ast f\sigma}_{L^2(\R^3)}\leq {\mathcal R^2_\phi} \norma{f}_{L^2(\R^2)}^2,
\end{equation}
where the optimal constant is given by 
\begin{equation}
\label{bestConstantS}
\mathcal{R}_{\phi}:=\sup_{0\neq f\in L^2(\R^2)} \frac{\norma{f\sigma\ast f\sigma}^{\frac12}_{L^2(\R^3)}}{\norma{f}_{L^2(\R^2)}}.
\end{equation}
\begin{define}
An {\bf extremizing sequence} for  inequality \eqref{2-convolution} is a sequence $\{f_n\}\subset L^2(\R^2)$ satisfying $\|f_n\|_{L^2(\R^2)}\leq 1$, such that 
$$\norma{f_n\sigma\ast f_n\sigma}_{L^2(\R^3)}\to {\mathcal R^2_\phi}, \textrm{ as } n\to\infty.$$
An {\bf extremizer} for inequality \eqref{2-convolution} is a nonzero function $f\in L^2(\R^2)$ which satisfies
$$ \norma{f\sigma\ast f\sigma}_{L^2(\R^3)}= {\mathcal R^2_\phi} \norma{f}_{L^2(\R^2)}^2.$$
\end{define}

\noindent The unperturbed case $\phi=0$ was treated by Foschi \cite{F}, who proved that extremizers for the corresponding extension inequality on the paraboloid $\Sigma_0$ are given by Gaussians, and computed
$\mathcal R_0=(\frac{\pi}2)^{\frac14}$. A key step in Foschi's program was the elementary but crucial observation that the convolution of projection measure on the two-dimensional 
paraboloid defines a constant function in the interior of its support,  see  \cite[Lemma 3.2]{F}, and Remark \ref{Foschi} below.  
Alternative approaches are available:
Hundertmark--Zharnitski \cite{HZ} base their analysis on a novel representation of the Strichartz integral, and Bennett et al. \cite{BBCH} identify a monotonicity property of such integrals under a certain quadratic heat-flow. 
These proofs rely on the large symmetry group enjoyed by the paraboloid. 
Perturbed paraboloids  $\Sigma_\phi$ with $\phi\neq 0$ no longer enjoy this special feature, and 
understanding the associated Fourier extension operator in sharp form is an important step towards the understanding of general manifolds with positive Gaussian curvature. 
This motivates our first main result.

\begin{teo}\label{main}
Let $\phi:\R^2\to\R$ be a nonnegative, twice continuously differentiable, strictly convex function, whose Hessian $H(\phi)$ satisfies one of the following conditions:
\begin{itemize}
 \item [(i)] $H(\phi)(y_0)=0$ for some $y_0\in\R^2$, or
 \item [(ii)] There exists a sequence $\{y_n\}\subset \R^2$ with $\ab{y_n}\to\infty$, such that  {$H(\phi)(y_n)\to0$}, as $n\to\infty$.
\end{itemize}
Let $\sigma$ denote the projection measure on the surface $\Sigma_\phi$.
Then the inequality
\begin{equation}\label{2-convolution-2}
 \norma{f\sigma\ast f\sigma}_{L^2(\R^3)}\leq {\mathcal R^2_\phi} \norma{f}_{L^2(\R^2)}^2
 \end{equation}
holds for every $f\in L^2(\R^2)$, and is sharp with optimal constant  given by
\begin{equation}\label{ValueS}
\mathcal R_\phi
=\Big(\frac \pi 2\Big)^{\frac14}.
\end{equation}
The sequence $\{f_n/\|f_n\|_{L^2}\}$ defined via 
\begin{equation}\label{extseq}
f_n(y):= \left\{ \begin{array}{ll}
e^{-n(\psi(y)-\psi(y_0)-\langle\nabla\psi(y_0),y-y_0\rangle)}, & \textrm{in case \emph{(i)},}\\
e^{-a_n(\psi(y)-\psi(y_n)-\langle\nabla\psi(y_n),y-y_n\rangle)},& \textrm{in case \emph{(ii)},}
\end{array} \right.
\end{equation}
where $\psi:=|\cdot|^2+\phi$ and $\{a_n\}$ is an appropriately chosen sequence, is  extremizing for inequality \eqref{2-convolution-2}.
Moreover, extremizers for inequality \eqref{2-convolution-2} do not exist.
\end{teo}

\noindent Let us briefly comment on the proof of Theorem \ref{main}. In order to compute the optimal constant $\mathcal R_\phi$ and to show that extremizers do not exist, we employ methods from \cite{OS,Q} that 
are based on 
Foschi's ideas \cite{F}, with a novel ingredient which we highlight below. The main steps are the following:
\begin{itemize}
 \item One shows that $\mathcal R_\phi^4\leq \norma{\sigma\ast\sigma}_{L^\infty}<\infty$.
 \item One exhibits an explicit sequence $\{f_n\}\subset L^2(\R^2)$ such that
 \[\liminf_{n\to\infty} \frac{\norma{f_n\sigma\ast f_n\sigma}_{L^2}^2}{\norma{f_n}_{L^2}^4}\geq \norma{\sigma\ast\sigma}_{L^\infty}.\]
 \item From the previous two steps, one concludes  $\mathcal R^4_\phi= \norma{\sigma\ast\sigma}_{L^\infty}$.
 \item One proves that the set  
 $\{(\xi,\tau)\in\R^{2+1}: (\sigma\ast\sigma)(\xi,\tau)=\norma{\sigma\ast\sigma}_{L^\infty}\}$ has Lebesgue measure zero.
 \item A careful review of Foschi's method then implies 
 \[\norma{f\sigma\ast f\sigma}_{L^2}<\mathcal R_\phi^2 \norma{f}_{L^2}^2,\]
 for every nonzero $f\in L^2(\R^2)$. In particular, extremizers do not exist.
\end{itemize}

\noindent The first and fourth steps above are based on a new comparison principle that translates into a pointwise inequality between convolution of projection measures on the perturbed surface $\Sigma_\phi$ and the  paraboloid $\Sigma_0$, respectively. It leads to the computation of the exact numerical value of the optimal constant $\mathcal R_\phi$. The comparison principle holds in all dimensions $d\geq 2$, and we state it precisely as follows.

\begin{teo}\label{comparison-convolution}
For $d\geq 2$, let $\phi:\R^d\to\R$ be a nonnegative, continuously differentiable, strictly convex function.  Let $\vphi=\ab{\cdot}^2$ and $\psi=\ab{\cdot}^2+\phi$.
Let $\sigma_0, \sigma$ denote the projection measures on the hypersurfaces $\Sigma_0, \Sigma_\phi$, respectively.
Then
\begin{equation}\label{comparisonineq}
\bigl(\sigma\ast\sigma\bigr)(\xi,2\psi({\xi}/2)+\tau)\leq 
\bigl(\sigma_0\ast\sigma_0\bigr)(\xi,2\vphi(\xi/2)+\tau),
\end{equation}
for every $\xi\in\R^d$ and $\tau>0$. Moreover, this inequality is strict for almost every point in the support of the measure $\sigma\ast\sigma$. 
\end{teo}

\noindent Certain related but distinct comparison principles have already proved useful in understanding the effect of global smoothing. See \cite{RS} for an instance  in which such a principle was used to derive new estimates for dispersive (and non-dispersive) equations from known ones, as well as an effective means to compare estimates for different equations. The link with optimal constants and extremizers for a broad class of smoothing estimates is established in \cite{BS}. \\

\noindent In the second part of the present paper, we adapt ideas from the concentration-compactness principle of Lions \cite{Li} to examine the behavior of general extremizing sequences for inequality \eqref{2-convolution-2}. Generally speaking, the theory of concentration-compactness has proved a very efficient tool to exhibit the precise mechanisms which are responsible for the loss of compactness in a variety of settings. In our concrete problem, extremizers fail to exist because extremizing sequences concentrate. Concentration can only occur at points where the convolution $\sigma\ast\sigma$ attains its maximum value, or at spatial infinity. To make these concepts precise, we introduce the relevant definitions. 
\begin{define}
 \label{DefConcentration}
A sequence $\{f_n\}\subset L^2(\R^2)$ {\bf concentrates at a point} $y_0\in\R^2$ if, for every $\eps,\rho>0$, there exists $N\in\N$ such that, for every $n\geq N$,
\[\int_{\{\ab{y-y_0}\geq\rho\}}\ab{f_n(y)}^2 \d y< \eps\norma{f_n}_{L^2(\R^2)}^2.\]
A sequence $\{f_n\}\subset L^2(\R^2)$ {\bf concentrates along a sequence} $\{y_n\}\subset\R^2$ if, for every $\eps,\rho>0$, there exists $N\in\N$ such that, for every $n\geq N$,
\[\int_{\{\ab{y-y_n}\geq\rho\}}\ab{f_n(y)}^2 \d y< \eps\norma{f_n}_{L^2(\R^2)}^2.\]
A sequence $\{f_n\}\subset L^2(\R^2)$ {\bf concentrates at infinity} if, for every $\eps,\rho>0$, there 
exists $N\in\N$ such that, for every $n\geq N$,
\[\int_{\{\ab{y}\leq\rho\}}\ab{f_n(y)}^2 \d y< \eps\norma{f_n}_{L^2(\R^2)}^2.\]
\end{define}

\noindent The following result holds under the general hypotheses of Theorem \ref{main}.

\begin{teo}
 \label{BehaviorExtSeq}
 Let $\phi:\R^2\to\R$ be a nonnegative, twice continuously differentiable, strictly convex function, whose Hessian satisfies condition \emph{(i)} or condition \emph{(ii)} from Theorem \ref{main}.
Then any extremizing sequence for inequality \eqref{2-convolution-2}
has a further subsequence
which either concentrates at some point $y_0\in\R^2$ satisfying $H(\phi)(y_0)=0$, 
 or concentrates at infinity.
\end{teo}
\vspace{.2cm}

\noindent It has long been understood that Tomas--Stein extension type inequalities are related to Strichartz estimates for linear partial differential equations of dispersive type.
To illustrate this connection in the present situation, consider the 
multiplier operator 
$$\widehat{M_\phi g}=\phi\, \widehat g$$
acting on Schwartz functions $g$, and the associated Schr\"odinger equation
\begin{equation}\label{higher_schodinger_eq}
   \begin{cases}
   i u_t+M_\phi u-\mu\Delta u=0,\quad \mu\geq 0,\\
     u(\cdot,0)=f\in L^2(\R^d),
  \end{cases} 
\end{equation}
whose solution can be written  as
\begin{equation}
 \label{solution-u}
 u(x,t)={\frac{1}{(2\pi)^d}}\int_{\R^d} \widehat f(y) e^{i\langle x, y\rangle}e^{it(\mu\ab{y}^2+\phi(y))}\d y,\;\;\;(x,t)\in\R^{d+1}.
\end{equation}
In the third part of the paper, we consider Strichartz inequalities for solutions of  equation
\eqref{higher_schodinger_eq} in the two-dimensional case $d=2$. In particular, we investigate the family of inequalities
\begin{equation}\label{Strichartz}
\norma{(\mu+\ab{\nabla}^2)^{\frac 14}e^{it(\phi(\ab{\nabla})-\mu\Delta)}f}_{L^4(\R^3)} 
\lesssim_{\mu,\phi}
\norma{f}_{L^2(\R^2)},
\end{equation}
and mostly focus on the particular instance of a quartic perturbation, $\phi=|\cdot|^4$. 
In this case, inequality \eqref{Strichartz} can be proved via the method of stationary phase together with the main theorem of \cite{KT}, see the remarks preceding \cite[Proposition 2.4]{JSS}, and \cite{BKS, KPV, Pa} for further details.
In the spirit of what had been done in the one-dimensional setting in \cite{JPS}, this instance of inequality \eqref{Strichartz} was refined in \cite{JSS}, with the goal of establishing a linear profile decomposition for a family of fourth order Schr\"odinger equations.
As a consequence, the authors of \cite{JSS} obtained a dichotomy result for the existence of extremizers in the cases $\mu\in\{0,1\}$, which by scaling extends to the general case $\mu\geq 0$, and can be summarized as follows: 
Either extremizers exist, or extremizing sequences exhibit a certain classical Schr\"odinger behavior.
See \cite[Theorems 4.1 and 4.2]{JSS} for a precise formulation of these results.
Along the way, the authors of \cite{JSS} obtained lower bounds for the norms of the corresponding Fourier extension operators.
 The methods we use to study the sharp bilinear convolution inequality \eqref{2-convolution-2} are robust enough to resolve this dichotomy, and to determine which situation actually happens.
In particular, we prove that
 extremizers exist if $\mu=0$, but fail to exist if $\mu=1$.
In the latter case, we also compute the operator norm exactly. 

\noindent To state our results precisely, let us start by considering the case of $\mu=0$ and $\phi=|\cdot|^4$.
Then inequality \eqref{Strichartz} can be restated with the help of the Fourier transform, here denoted by $\mathcal{F}$, as
\begin{equation}
\label{hom-Sobolev--Strichartz}
\norma{\mathcal{F}(f\ab{\cdot}^{\frac12}\nu)}_{L^4(\R^3)}\lesssim\norma{f}_{L^2(\R^2)},
\end{equation}
where the measure $\nu$ is given by $\nu(y,t)=\ddirac{t-\ab{y}^4} \d y\d t$.
By Plancherel's Theorem, inequality \eqref{hom-Sobolev--Strichartz} can be rewritten in sharp form as
\begin{equation}
\label{hom_conv_Sobolev_Strichartz_Intro}
\norma{f{|\cdot|^{\frac 12}}\nu\ast f{|\cdot|^{\frac 12}}\nu}_{L^2(\R^3)}\leq \mathcal Q^2\norma{f}_{L^2(\R^2)}^2, 
\end{equation}
with optimal constant $\mathcal Q$. The following result  should be compared to \mbox{\cite[Theorem 4.1]{JSS}.}

\begin{teo}
	\label{extPureQuarticIntro}
 The optimal constant for inequality \eqref{hom_conv_Sobolev_Strichartz_Intro} satisfies the bounds
	\begin{equation}\label{sharpBoundsforQ}
	\frac{\pi}{4\sqrt{3}}< \mathcal Q^4<\frac{\pi}{4}.
	\end{equation}
	Moreover, there exists an extremizer for inequality \eqref{hom_conv_Sobolev_Strichartz_Intro}.
\end{teo}

\noindent Still taking $\phi=|\cdot|^4$, let us now consider the case of $\mu=1$. 
Then inequality \eqref{Strichartz} can be restated as
\begin{equation}\label{quartic0_Intro}
\norma{\mathcal{F}(f(1+\ab{\cdot}^2)^{\frac14}\sigma)}_{L^4(\R^3)}\lesssim\norma{f}_{L^2(\R^2)}, 
\end{equation}
where the measure $\sigma$ is given by $\sigma(y,t)=\ddirac{t-\ab{y}^2-\ab{y}^4} \d y\d t$.
By Plancherel's Theorem, inequality \eqref{quartic0_Intro} can be rewritten in sharp form as
\begin{align}\label{quartic_Intro}
\norma{f\sqrt{w}\sigma\ast f\sqrt{w}\sigma}_{L^2(\R^3)}&\leq 
\mathcal{S}^2\norma{f}_{L^2(\R^2)}^2,
\end{align}
with weight  $w=(1+\ab{\cdot}^2)^{\frac 12}$ and optimal constant $\mathcal{S}$. 
The following result  is a special case of Theorem \ref{bestConstantQuartic} below.

\begin{teo} \label{bestConstantQuarticIntro}
	The value of the optimal constant for inequality  \eqref{quartic_Intro} is given by
	$\mathcal{S}^4=\frac{\pi}{2}.$ 
	 Moreover, extremizers for inequality \eqref{quartic_Intro} do not exist, and extremizing sequences concentrate at the origin. 
\end{teo}

\noindent In particular, Theorems \ref{extPureQuarticIntro} and \ref{bestConstantQuarticIntro} imply that $\mathcal{Q}^4<\mathcal{S}^4=\frac{\pi}2$. The  shape of (the Fourier transform of) a general extremizing sequence for inequality \eqref{quartic_Intro} 
is then given by \cite[Theorem 4.2]{JSS} and the remarks following it.
Furthermore, as mentioned in \cite{JSS}, it is of interest to extend the analysis to more general perturbations of the Schr\"odinger equation.
Our methods allow to make progress in a number of previously untreated cases, and we comment on this in Remark \ref{remark64} and \S \ref{sec:OtherPP} below.
\\

\noindent Our results complement the recent, vast and very interesting body of work concerning sharp Fourier extension and Strichartz estimates. In addition to the works previously cited in this introduction, see \cite{COS, CS, CS2, F2, FLS} for results in sharp Fourier extension theory on spheres, and  \cite{C, FVV, FVV2, Han, OS2, R} for other instances.
\\

\noindent{\bf Overview.}  The paper is organized as follows. In Chapter \ref{sec:ScalingConvol}, we 
briefly comment on the model case of a pure power perturbation of the paraboloid, and derive a useful integral formula for the convolution of 
projection measure on a generic convex perturbation of the two-dimensional paraboloid. Chapter \ref{sec:Comparison} is the technical heart of the 
first part of the paper, and is devoted to the aforementioned comparison principle. In particular, we prove Theorem \ref{comparison-convolution}, and 
briefly remark on possible extensions of this result to $n$-fold convolutions if $n\geq 3$. The proof of Theorem \ref{main} is presented in Chapter 
\ref{sec:NoExt}. We discuss the behavior of general extremizing sequences in Chapter \ref{sec:BehaviorExtSeq}. In particular, we establish a 
precise form of the geometric principle that distant caps interact weakly, show some auxiliary results of concentration-compactness flavor, and prove 
Theorem \ref{BehaviorExtSeq}.
Finally, we deal with sharp Strichartz inequalities in Chapter \ref{sec:Strichartz}.
We treat the case of quartic perturbations in \S \ref{sec:QuarticPert}, establishing a generalization of Theorem \ref{bestConstantQuarticIntro}. We study the convolution of projection measure associated to pure powers in \S \ref{sec:ConvPP}, and use this knowledge to tackle the case of the pure quartic in \S \ref{sec:PureQuartic}, establishing Theorem \ref{extPureQuarticIntro}, and of other pure powers in \S \ref{sec:OtherPP}.  
\\

\noindent{\bf A word on forthcoming notation.} 
The usual inner product between vectors $x,y\in\R^d$ will continue to be denoted by $\langle x,y\rangle$. This is to clarify the distinction from the $d\times d$ matrix obtained as the matrix product between the vector $x$ and the transpose of the vector $y$, denoted $x\cdot y^T$.
The usual matrix product between a $d\times d$ matrix $A$ and a vector $x\in \R^d$ will likewise be indicated by $A\cdot x$. The $d\times d$ identity matrix will be denoted by $I_d$, or simply by $I$ if no confusion arises. The open ball of radius $r>0$ centered at $x\in \R^d$  will be denoted by $B_r(x)$. If $x=0$, then we will simply write $B_r$ instead of $B_r(0)$. The corresponding closed balls will be denoted by $\bar B_r(x)$ and $\bar B_r=\bar B_r(0)$, respectively.  
The alternative notation for the Fourier transform $\mathcal{F}(f)=\widehat{f}$ will occasionally be used.
Finally, $\mathbbm{1}_E$ will stand for the indicator function of a given subset $E\subset\R^d$, and the complement of $E$ will at times be denoted by $E^\complement$.

\section{On scaling and convolutions}\label{sec:ScalingConvol}

\subsection{An explicit example}
For $d\geq 1$, $a> 0$ and $p>2$, consider the family of Fourier extension operators associated to certain polynomial perturbations of the paraboloid equipped with projection measure, given by
\begin{equation}\label{DefTa}
T_a(f)(x,t)=\int_{\R^d} f(y) e^{-i\langle x, y\rangle}e^{-it(\ab{y}^2+a\ab{y}^p)} \d y,\;\;\;(x,t)\in\R^{d+1}.
\end{equation}
The family $\{T_a\}$ enjoys the following scaling property. Given $a,b>0$, let 
$\la=(\frac ba)^{\frac1{p-2}}.$
Changing variables $y\rightsquigarrow\la y$ in the 
integral \eqref{DefTa}, we see that 
\begin{align*}
 T_a(f)(x,t)
 &=\la^{\frac d2}\int_{\R^d} f_\la( y) e^{-i\langle\la x, y\rangle}e^{-i\la^2t(\ab{y}^2+b\ab{y}^p)}\d y
 =\lambda^{\frac d2}T_b (f_\lambda)(\lambda x,\lambda^2 t),
\end{align*}
where the rescaled function $f_\la(y)=\la^{\frac d2}f(\la y)$ satisfies $\norma{f_\la}_{L^2}=\norma{f}_{L^2}$. It follows that 
$$\norma{T_a(f)}_{L^{2+\frac 4d}(\R^{d+1})}=\norma{T_b(f_\la)}_{L^{2+\frac 4d}(\R^{d+1})},$$ 
and therefore
\[\sup_{0\neq f\in L^2(\R^d)}\frac{\norma{T_a(f)}_{L^{2+\frac 4d}(\R^{d+1})}}{\norma{f}_{L^2(\R^d)}}=\sup_{0\neq f\in 
L^2(\R^d)}\frac{\norma{T_b(f)}_{L^{2+\frac4d}(\R^{d+1})}}{\norma{f}_{L^2(\R^d)}}.\]
From this identity, we conclude that optimal constants are independent of $a$, and that extremizers exist for some value of $a>0$ if and only if they exist for every value of $a>0$. If extremizers exist for one value of $a>0$, then the simple dilation indicated above produces an extremizer for any other value of $a>0$. 
Theorem \ref{main} provides a refinement of this rudimentary analysis in the two-dimensional case. In particular, it states that the optimal constant is also independent of $p$, and that extremizers \mbox{do not exist}.

\subsection{Convolution of singular measures}\label{sec:Formulas}
The goal of this section is to exhibit an explicit formula for the convolution of
 projection measure on perturbed paraboloids.
For the sake of concreteness, we limit our discussion to the two dimensional case $d=2$.
See \cite[Lemma 3.1]{BM} for a formula in the same spirit of identity \eqref{ConvolutionFormula} below. 

\begin{prop}\label{ConvolFormulaProp}
 Let $\psi:=|\cdot|^2+\phi$, where $\phi\geq 0$ is a convex $C^2(\R^2)$ function. Let $\sigma$ denote projection measure on the surface $\Sigma_\phi$. Then the following assertions hold for the convolution measure $\sigma\ast\sigma$:
 \begin{itemize}
 \item[(a)] It is absolutely continuous with respect to Lebesgue measure on $\R^3$.
 \item[(b)] Its support is given by
 $$\supp(\sigma\ast\sigma)=\{(\xi,\tau)\in\R^{2+1}: \tau\geq 2\psi(\xi/2)\}.$$
 \item[(c)] Its Radon--Nikodym derivative, also denoted by $\sigma\ast\sigma$, is given by the formula
\begin{equation}\label{ConvolutionFormula}
(\sigma\ast\sigma)(\xi,\tau)
=
\int_{\mathbb{S}^1}
\Big(\int_{-1}^1
\langle \omega, 
H(\psi)(\xi/ 2+t\alpha(\xi,\tau, \omega)\omega)
\cdot \omega\rangle\;
\d t\Big)^{-1}
\d\mu_\omega,
\end{equation}
provided $\tau>2\psi(\xi/2)$. 
Here, the measure $\mu$ denotes arc length measure on the unit circle $\mathbb{S}^1$, and the function $\alpha$ is given by 
\begin{equation}\label{alphaDef}
\alpha(\xi,\tau, \omega)={\sqrt{\tau/2-\psi(\xi/2)}}\la\Big({\sqrt{\tau/2-\psi(\xi/2)}}\omega\Big),
\end{equation}
where the function $\la$ is implicitly defined via identity \eqref{implicitlambda} below.
 \item[(d)] The convolution $\sigma\ast\sigma$ defines a continuous function of the variables $\xi,\tau$ in the interior of its support. It extends continuously to the boundary of the support, with values given by
\begin{align}\label{bdryvaluesDet}
(\sigma\ast\sigma)(\xi,2\psi(\xi/2))
={\frac{\pi}{\sqrt{\det(H(\psi)(\xi/2))}}}.
 \end{align}
\end{itemize} 
\end{prop}

\begin{remark}\label{Foschi} In the special case $\phi=0$, the Hessian of $\psi$ is a constant multiple of the identity matrix, and formula \eqref{ConvolutionFormula} recovers the result from \cite[Lemma 3.2]{F} for the convolution of projection measure $\sigma_0$ on the two-dimensional paraboloid $\Sigma_0\subset\R^3$: For $\tau> {|\xi|^2/2}$,
$$(\sigma_0\ast\sigma_0)(\xi,\tau)=\int_{\mathbb{S}^1}\Big(\int_{-1}^1\langle \omega, 2\omega\rangle \d t\Big)^{-1}\d\mu_\omega=\frac{\pi}{2}.$$
\end{remark}

\begin{proof}[Proof of Proposition \ref{ConvolFormulaProp}]
The absolute continuity of $\sigma\ast\sigma$ with respect to Lebesgue measure follows in the same way as in the proof of \cite[Lemma 3.1 (b)]{BM}, 
with minor modifications only. We provide the details for the convenience of the reader. In order to show that the pairing
$\langle\sigma\ast\sigma,\mathbbm{1}_E\rangle=0$ for each set $E$ of Lebesgue measure zero in $\R^3$, set $y=(y_1,y_2)$, $z=(z_1,z_2)$, and change variables $t_j=y_j+z_j$, $s_j=y_j-z_j$ ($j=1,2$) 
to get
\begin{align*}
\langle\sigma\ast\sigma,\mathbbm{1}_E\rangle
&=\int_{(\R^2)^2} \mathbbm{1}_E(y+z,\psi(y)+\psi(z)) \d y\d z\\
&=\frac 1 4\int_\R\Big(\int_{\R^2} \int_\R \mathbbm{1}_E(t_1,t_2,F_t(s_1,s_2)) \d s_1 \d t \Big)\d s_2,
\end{align*}
where the function $F_t$ is defined as
$$F_t(s)=\psi\Big(\frac{t+s}{2}\Big)+\psi\Big(\frac{t-s}{2}\Big).$$
The key observation is that $F_t(s_1,s_2)$ is a strictly convex function of $s_1$ for each fixed $s_2$ and $t$. The change of variables $s_1\mapsto u$ given by the (at most 2--to--1) map $u=F_t(s_1,s_2)$ shows that the triple integral in $(s_1,t)$ is zero (for each $s_2$) since $E$ is a Lebesgue null set. This establishes (a).

\noindent For (b), consider vectors $y,y'\in\R^2$, and note that 
$$\xi:=y+y'\textrm{ and }\tau:=\psi(y)+\psi(y')$$
satisfy $\tau\geq 2\psi(\xi/2)$ because the function $\psi$ is convex. For the reverse inclusion, let $(\xi,\tau)\in \R^{2+1}$ be given, such that $\tau\geq 2\psi(\xi/2)$. We want to find $y,y'\in\R^2$, \mbox{such that}
$$y+y'=\xi\textrm{ and }\psi(y)+\psi(y')=\tau.$$
It is enough to find $y$ such that $\psi(y)+\psi(\xi-y)=\tau$, for then $y'=\xi-y$. Note that $\psi(y)+\psi(\xi-y)\geq 2\psi(\xi/2)$ by convexity of $\psi$, with equality if $y=\xi/2$. Moreover,
$$\psi(y)+\psi(\xi-y)\to\infty,\textrm{ as }|y|\to\infty.$$
The function $y\mapsto \psi(y)+\psi(\xi-y)$ is continuous because $\psi$ is convex, and the result follows from applying the Intermediate Value Theorem in the appropriate direction.

\noindent We now come to part (c). Let $(\xi,\tau)$ be such that $\tau> 2\psi(\xi/2)$. Fubini's Theorem and  a simple change of variables yield\footnote{For a treatment of integration on manifolds using delta calculus, see \cite[Appendix A]{FOS}.} 
\begin{align}
(\sigma\ast\sigma)(\xi,\tau)
&=\int_\R\Big(\int_{\R^2}\ddirac{\tau-t-\psi(\xi-y)}\ddirac{t-\psi(y)}\d y\Big)\d t\notag\\
&=\int_{\R^2}\ddirac{\tau-\psi( \xi/2 +y)-\psi( \xi/2 -y)}\d y.\label{DeltaCalc}
\end{align}
We would like to perform another change of variables  $y=T(w)$, where $T(w)=\lambda w$, and $\lambda=\lambda(w)> 0$ is an 
implicit real-valued function of $w$ which takes only positive values, and is defined via
\begin{equation}\label{implicitlambda}
\psi( \xi/2+\lambda w)+\psi( \xi/2-\lambda w)
=2|w|^2+{2\psi(\xi/2)}.
\end{equation}
For fixed $\xi$, a unique positive solution $\la=\la(w)$ exists if ${w\neq 0}$. By the Implicit Function Theorem, 
equation 
\eqref{implicitlambda} defines $\la$ as a continuously differentiable function of $w$, as long as the derivative of the map
\[\lambda\mapsto
\psi( \xi/2+\lambda w)+\psi( \xi/2-\lambda w)\]
is nonzero. In view of the strict convexity of the function $\psi$, this is indeed the case if $\la>0$. Further details in a more general context will be 
provided in the course of the proof of Lemma \ref{transformation} below.
Since the function $\lambda$ is continuously differentiable and $T(w)=\lambda(w)w$, we have that 
\begin{equation}\label{1Tprime}
T'(w)=\lambda I+w\cdot (\nabla\la)^T,
\end{equation}
where $I$ denotes the $2\times 2$ identity matrix, the gradient is taken with respect to $w$, and the term $w\cdot (\nabla\la)^T$ stands  for the $2\times 2$ matrix obtained as the product of the vector $w$ and the transpose of the gradient $\nabla\la$ (seen as a vector in $\R^2$). 
To compute the gradient $\nabla\la$, note that implicit differentiation of \eqref{implicitlambda} with respect to $w$ yields
\begin{equation}\label{2Tprime}
(T')^T(w)\cdot u=4w,\textrm{ where }u=u(w,\xi):=\nabla\psi( \xi/2+\lambda w)-\nabla\psi( \xi/2-\lambda w).
\end{equation}
From \eqref{1Tprime} and \eqref{2Tprime}, it follows that 
\begin{equation}\label{GradientLambda}
\nabla\la=\frac{4w-\la u}{\langle w, u\rangle}.
\end{equation}
One easily computes 
$$\det T'(w)
=\det (\lambda I+w\cdot (\nabla\la)^T)
=(1+\la^{-1}\langle w,\nabla \la\rangle)\det(\la I),
$$
and identity \eqref{GradientLambda} then implies
\begin{equation}\label{detTprime}
\det T'(w)=\frac{4|w|^2\la(w)}{\langle w, u(w,\xi)\rangle}.
\end{equation}
Note that this is a nonnegative quantity because of the strict convexity of $\psi$. Going back to the integral expression for $\sigma\ast\sigma$, changing variables as announced, and switching to polar coordinates, yields
\begin{align*}
(\sigma\ast\sigma)(\xi,\tau)
&=
\int_{{\R^2}}\ddirac{\tau-{2\psi(\xi/2)}-2|w|^2} \det T'(w) \d w\\
&=
\int_{{0}}^\infty\ddirac{\tau-{2\psi(\xi/2)}-2 r^2} \Big(\int_{\mathbb{S}^1}\det T'(r\omega) \d\mu_\omega\Big) r \d r,
\end{align*}
where $\mu$ denotes arc length measure on the unit circle $\mathbb{S}^1\subset \R^2$. Using expression \eqref{detTprime} for the Jacobian factor $\det T'$, 
changing variables $2r^2=s$, and appealing to Fubini's theorem, we have that
\begin{align*}
(\sigma\ast\sigma)(\xi,\tau)
=
\int_{\mathbb{S}^1}
\Big(\int_{{0}}^\infty \ddirac{\tau-{2\psi(\xi/2)}-s}\frac{\sqrt{s/2}\la(\sqrt{s/2}\omega)}{\langle\omega, u(\sqrt{s/2} \omega,\xi)\rangle} \d s\Big)
\d\mu_\omega.
\end{align*}
Evaluating the inner integral,
$$(\sigma\ast\sigma)(\xi,\tau)
=
\int_{\mathbb{S}^1}
\frac{{\sqrt{\tau/2-\psi(\xi/2)}}\la({\sqrt{\tau/2-\psi(\xi/2)}}\omega)}{\langle\omega, u({\sqrt{\tau/2-\psi(\xi/2)}} \omega,\xi)\rangle} \d\mu_\omega.
$$
Defining the function $\alpha=\alpha(\xi,\tau,\omega)$ as in \eqref{alphaDef}, and 
recalling the expression in \eqref{2Tprime} for the vector $u$, 
\begin{equation}\label{preHessian}
(\sigma\ast\sigma)(\xi,\tau)
=
\int_{\mathbb{S}^1}
\Big\langle \omega, \frac{\nabla \psi(\xi/2+\alpha\omega)-\nabla \psi(\xi/2-\alpha\omega)}{\alpha}\Big\rangle^{-1}
\d\mu_\omega.
\end{equation}
 Formula \eqref{ConvolutionFormula} now follows from the Fundamental Theorem of Calculus. 
 
\noindent  As for part (d), the continuity in the interior of the support follows from an inspection of representation formula 
\eqref{ConvolutionFormula}, after recalling the fact that the function $\la$ is continuous. The boundary value is obtained by noting that, for each $\omega\in\mathbb{S}^1$, the function  
$\alpha(\xi,\tau,\omega)$ tends to $0$ as $(\xi,\tau)$ approaches the boundary of the support from 
its interior, since the function $\la$ satisfies $0\leq\la\leq 1$. This yields
$$(\sigma\ast\sigma)(\xi,2\psi(\xi/2))=\frac 12\int_{\mathbb{S}^1} \frac{1}{\langle\omega,H(\psi)(\xi/2)\cdot\omega\rangle}\d\mu_\omega,$$
from which  identity
\eqref{bdryvaluesDet} follows by using an orthonormal basis of $\R^2$ consisting of 
eigenvectors of the Hessian matrix $H(\psi)(\xi/2)$. The proof is now complete.
\end{proof}

\begin{remark} 
Identity \eqref{preHessian} already implies a weak form of the comparison principle (Theorem \ref{comparison-convolution}) in the two-dimensional case.
Analogous reasoning leads to similar formulae for higher dimensional hypersurfaces. This is one of the motivations for the next chapter, which shares some features with the proof of Proposition \ref{ConvolFormulaProp}. 
However, the analysis there seems more flexible, and may be adaptable to other situations \mbox{as well}.
\end{remark}

\section{A comparison result}\label{sec:Comparison}

This chapter is devoted to the proof of Theorem \ref{comparison-convolution}, which holds in dimensions $d\geq 2$.
Before stating the technical lemmata that will be used in its proof, let us consider two convex functions $\psi,\varphi:\R^d\to\R$. 
 Given $\xi,y\in\R^d$, define the following auxiliary functions of one real variable:
\begin{align}
g(t)&:=\psi(\xi/2-ty)+\psi(\xi/2+ty)-2\psi(\xi/2),\label{gdef}\\
 h(t)&:=\vphi(\xi/2-ty)+\vphi(\xi/2+ty)-2\vphi(\xi/2).\label{hdef}
\end{align}
Note that $g=h\equiv 0$ if $y=0$. Some properties of the functions $g,h$ in a useful special case are collected in the following lemma.

\begin{lemma}\label{convexity}
Let $\psi,\varphi:\R^d\to\R$ be differentiable, convex functions, such that their difference $\psi-\varphi$ is also convex.
 Given $\xi,y\in\R^d$, define the functions $g, h$ as above. Then:
 \begin{enumerate}
 \item[(a)] $g(t)\geq h(t)\geq 0$, for every $t\in\R$.
  \item[(b)] The functions $g$ and $h$ are convex.
  \item[(c)] $g'(0)=h'(0)=0$.
  \item[(d)] If $\psi$ is strictly convex and $y\neq 0$, then $g$ attains its unique global minimum at $t=0$.
  \item[(e)] If $\psi$ is strictly convex and $y\neq 0$, then there exists a unique nonnegative $\la=\la(y,\xi)$ such that
  \[h(1)=g(\la),\]
and  moreover $\la\in[0,1]$.
 \item[(f)] If $h(1)>0$, then $\la>0$. If $h(1)<g(1)$, then $\la<1$.
 \end{enumerate}
\end{lemma}

\begin{proof}
The inequality $h\geq 0$ follows from the (midpoint) convexity of the function $\varphi$. The inequality $g\geq h$ follows from the (midpoint) convexity of the function $\psi-\vphi$. This establishes (a).
Statement (b) is a consequence of the following two general facts: sums of convex functions are convex, and restrictions of convex functions to lines are convex.
Differentiability of the functions $g,h$ follows from that of $\psi,\vphi$. In light of (a), both $g$ and $h$ attain a (local, and therefore global) minimum at $t=0$ since $g(0)=h(0)=0$, and (c) follows. 
Further notice that $g$ is strictly convex if $\psi$ is strictly convex, provided $y\neq 0$. Since a strictly convex function can have at most one global minimum, (d) follows from (c).
We now consider statement (e). Since $g$ is continuous and $g(0)\leq h(1)\leq g(1)$,  the 
Intermediate Value Theorem ensures the existence of  
$\la\in[0,1]$ such that $h(1)=g(\la)$. There exists no  $\la$ in the interval $(1,\infty)$ with the same property because $g$ is strictly convex, and therefore 
$g(t)>g(1)$ if $t>1$. The uniqueness of $\la$ also follows from the strict convexity of $g$. Statement (f) is immediate, and this concludes the proof of the lemma.
\end{proof}

\noindent Henceforth we restrict attention to continuously differentiable functions $\psi, \vphi$ which are strictly convex, and introduce two sets which will play a  role in the proof of Theorem
 \ref{comparison-convolution}. Given $\xi\in\R^d$ and 
$c\in\R$, define the {\it $\psi$-ellipsoid} 
as
\begin{align}
 \label{psi-ellipsoid}
 \mathcal E_\psi(\xi,c):=\{y\in\R^d:\psi(\xi/2-y)+\psi(\xi/2+y)-2\psi(\xi/2)=c\},
 \end{align}
and similarly for the {\it $\vphi$-ellipsoid} $\mathcal E_\vphi(\xi,c)$. We will abuse notation mildly by occasionally referring to these sets simply as ``ellipsoids''.
The sets  $\mathcal E_\psi(\xi,c)$ and $\mathcal E_\vphi(\xi,c)$ are non-empty provided $c\geq 0$, and codimension 1 hypersurfaces if $c>0$. 
This claim requires a short justification which goes as follows. Since the function $\psi$ is differentiable and strictly convex, its gradient $\nabla \psi$ is a strictly monotone mapping, in the sense that
$$\langle\nabla\psi(x)-\nabla\psi(x'),x-x'\rangle>0,\textrm{ for every }x\neq x',$$ 
see, for instance, \cite[p.~112]{HUL}. As a consequence, any positive number $c>0$ is a regular value of the function $F_\psi:\R^d\to\R$, defined via
$$y\mapsto F_\psi(y)=\psi(\xi/2-y)+\psi(\xi/2+y)-2\psi(\xi/2),$$
and the claim follows for the ellipsoid $\mathcal E_\psi(\xi,c)=F_\psi^{-1}(c)$. The assertion for $\vphi$ can be verified in an identical way. Further note that, for each fixed $\xi\in\R^d$, the disjoint union of the ellipsoids $\mathcal E_\psi(\xi,c)$ as the parameter $c\geq 0$ ranges over the nonnegative real numbers equals the whole of $\R^d$, and similarly for $\vphi$.

\noindent After these preliminary observations, define the transformation 
\begin{equation}
 \label{transformation-T}
T:\R^d\setminus\{0\}\to\R^d,\;\;\; T(y)=\la(y,\xi)y,
\end{equation}
where $\la(y,\xi)$ is given by part (e) of Lemma \ref{convexity}. In other words, the real number $\la=\la(y,\xi)$ is the unique 
nonnegative solution of the equation
\begin{equation}\label{ImplicitIdentity}
\vphi(\xi/2-y)+\vphi(\xi/2+y)-2\vphi(\xi/2)=\psi(\xi/2-\la y)+\psi(\xi/2+\la y)-2\psi(\xi/2).
\end{equation}
 Relevant properties of the transformation $T$ are 
recorded in the next result.

\begin{lemma}\label{transformation}
 Let $\psi,\varphi:\R^d\to\R$ be continuously differentiable, strictly convex functions 
  with a convex difference $\psi-\varphi$.
Let $\xi\in \R^d$ be given, and consider the transformation $T$  given by \eqref{transformation-T}. Then:
 \begin{enumerate}
  \item[(a)] $T$ is injective. 
  \item[(b)] $T$ is continuously differentiable.
  \item[(c)] If $T'(y)$ denotes the Jacobian matrix of $T$ at a point $y\neq 0$, then 
  \begin{equation}\label{detT'formula}
 \det T'(y)
 =\la(y)^{d-1}\frac{\langle\nabla\vphi(\xi/2+y)-\nabla\vphi(\xi/2-y), y\rangle}{\langle\nabla\psi(\xi/2+T(y))-\nabla\psi(\xi/2-T(y)), y\rangle}.
 \end{equation} 
  \item[(d)] $T$ defines a bijection from the ellipsoid $\mathcal E_\varphi(\xi,c)$ onto the ellipsoid
$\mathcal E_\psi(\xi,c)$, for every 
$c>0$.
 \end{enumerate}
\end{lemma}
\begin{proof}
 To prove (a), let us consider nonzero vectors $y,z\in\R^d$ such that $T(y)=T(z)$. 
Then
 \[\psi(\xi/2-\la(y) y)+\psi(\xi/2+\la(y) y)
 =\psi(\xi/2-\la(z) z)+\psi(\xi/2+\la(z)z),\]
  where, for  notational convenience, we have dropped the 
dependence of $\la$ on $\xi$. 
 This implies
 \[\vphi(\xi/2-y)+\vphi(\xi/2+y)=\vphi(\xi/2-z)+\vphi(\xi/2+z).\]
 Since  $y=rz$ for some $r>0$, and the function $t\mapsto \vphi(\xi/2-tz)+\vphi(\xi/2+tz)$ is strictly increasing on $(0,\infty)$, 
 we obtain $r=1$. This means  $y=z$, as desired. 

\noindent Property (b) will follow from the Implicit Function Theorem, after showing that the derivative of the map 
$t\mapsto g(t)=\psi(\xi/2-ty)+\psi(\xi/2+ty)-2\psi(\xi/2)$
is nonzero for each $\xi,y\in\R^d$ with $y\neq 0$, provided $t>0$. 
This derivative equals 
\[g'(t)=\langle\nabla\psi (\xi/2+ty)-\nabla\psi(\xi/2-ty), y\rangle,\]
which is nonzero because of the strict convexity of $\psi$. Indeed, in the proof of Lemma \ref{convexity} we have already argued that $g$ is a strictly convex $C^1$ function which attains its unique global minimum at $t=0$, 
hence $g'(t)>0$ for every $t>0$.
Alternatively, recall that the gradient $\nabla \psi$ is a strictly monotone mapping.

\noindent To verify (c), we compute the Jacobian matrix of $T$ in an analogous way to what was  done in the proof of Proposition \ref{ConvolFormulaProp}. Implicit differentiation with respect to the variable $y$ of identity \eqref{ImplicitIdentity} with $\la=\la(y)$ yields
\[(\la I+\nabla \la\cdot  y^T)\cdot u=v,\]
where 
the vectors $u,v\in\R^d$ are defined by
\begin{align*}
u&=\nabla\psi (\xi/2+T(y))-\nabla\psi(\xi/2-T(y)),\\
v&=\nabla\vphi (\xi/2+y)-\nabla\vphi(\xi/2-y).
\end{align*}
For $y\neq 0$, it follows that
$$\nabla \la=\frac{v-\la u}{\langle u, y\rangle},$$
where the denominator $\langle u, y\rangle$ is strictly positive because the gradient $\nabla\psi$ is strictly monotone and the vector $T(y)$ is collinear with $y$.
Using this together with the Matrix Determinant Lemma, we arrive at identity \eqref{detT'formula}:
$$\det T'(y)=\det(\la I+\nabla\la\cdot y^T)
=\det(\lambda I)(1+\la^{-1} \langle y, \nabla\la\rangle)
=\la^{d-1}\frac{\langle v, y\rangle}{\langle u, y\rangle}.$$

\noindent We finally turn to (d). That the transformation $T$ has the desired mapping properties from $\mathcal E_\varphi$ into $\mathcal E_\psi$ follows from the defining identity \eqref{ImplicitIdentity}. In view of (a), the restriction of $T$ to the set $\mathcal E_\varphi$ is an injective map. So we are left with verifying surjectivity. The previous considerations show that, given $c>0$ and $z\in \mathcal E_\psi(\xi,c)$, it suffices to find {\it any} vector $y\in\R^d$ for which $T(y)=z$ (for such $y$ will then necessarily belong to $\mathcal E_\vphi(\xi,c)$). But $T$ is a continuous map which preserves rays emanating from the origin, such that $|Ty|\leq |y|$ for every $y\neq 0$, and 
$$\lim_{|y|\to\infty} |Ty|=\infty.$$
The result  follows from the Intermediate Value Theorem.
 \end{proof}

\noindent Recall that $\ab{T(y)}\leq \ab{y}$, for every $y\neq 0$. We would like to argue that the transformation $T$ is contractive in the sense that $|\det 
T'|< 1$. Unfortunately, an explicit computation involving the example $\vphi(x)=|x|^4$ and $\psi(x)=|x|^2+|x|^4+|x|^6$  reveals that,  perhaps 
unintuitively, one should not expect that to be the case in general. We will be interested  in convex perturbations of the paraboloid, and so the following 
result will suffice for our purposes.
 
 \begin{lemma}\label{contractive}
 Let $d\geq 2$.
 Let $\vphi=|\cdot|^2$ and $\psi=|\cdot|^2+\phi$, where $\phi\geq0$ is a strictly convex  $C^1(\R^d)$ function.
Let $\xi\in \R^d$ be given, and consider the transformation $T$  given by \eqref{transformation-T}.
 Then 
\begin{equation}\label{<1}
|\det T'(y)|< 1, \textrm{ for every }y\neq 0.
\end{equation}
 \end{lemma}
 
 \begin{proof}
 Fix $y\neq 0$.
 For the particular choices of $\psi,\vphi$ as in the statement of the lemma, define real-valued functions $g, h$ via identities \eqref{gdef} and \eqref{hdef}.
In this case,  $h'(t)=4|y|^2t$, a homogenous function of degree 1.
Identity \eqref{detT'formula} then implies 
 \begin{equation}\label{hmgDetT}
 \det T'(y)=\lambda(y)^{d-1}\frac{h'(1)}{g'(\lambda(y))}=\lambda(y)^{d-2}\frac{h'(\lambda(y))}{g'(\lambda(y))}.
 \end{equation}
 We have already argued that $g-h$ is a nonnegative,  differentiable, strictly convex function satisfying $(g-h)(0)=0$ and $(g-h)'(0)=0$. It follows that $(g-h)'(t)> 0$ for every $t> 0$, which means that the fraction on the right-hand side of identity \eqref{hmgDetT} is strictly less than 1 as long as $\lambda(y)>0$. That this is indeed the case follows from part (f) of Lemma \ref{convexity}, since $h(1)=2|y|^2>0$. The proof is finished by noting that $\la(y)\leq 1$ and $d\geq 2$ together imply $\la(y)^{d-2}\leq 1$.
\end{proof}
 
\noindent We have now collected all the ingredients needed to prove Theorem \ref{comparison-convolution}.

\begin{proof}[Proof of  Theorem \ref{comparison-convolution}]
As in the proof of Proposition \ref{ConvolFormulaProp}, the  convolutions can be written as
\begin{align}
 (\sigma\ast\sigma)(\xi,\tau)&=\int_{\R^d}\ddirac{\tau-\psi(\xi/2-y)-\psi(\xi/2+y)}\d y,\label{PsiConvolution}\\
(\sigma_0\ast\sigma_0)(\xi,\tau)&=\int_{\R^d}  \ddirac{\tau-\vphi(\xi/2-y)-\vphi(\xi/2+y)}\d y.\notag
\end{align}
A straightforward adaptation of the arguments there shows that the convolution $\sigma\ast\sigma$ is supported on the region  $\{(\xi,\tau):\tau\geq 
2\psi(\xi/2)\}$.  Since $\phi\geq 0$, this region is contained in the support of the convolution
$\sigma_0\ast\sigma_0$, i.e., the set $\{(\xi,\tau):\tau\geq 
2\vphi(\xi/2)\}$. 

\noindent For each fixed $\xi\in\R^d$,  consider the transformation $T$ given by \eqref{transformation-T}, which by  Lemma \ref{transformation} maps the ellipsoid 
$\mathcal E_\vphi (\xi,\tau)$
bijectively onto $\mathcal E_\psi (\xi,\tau)$,  for every $\tau>0$.
Changing variables $y\rightsquigarrow Ty$ in the expression \eqref{PsiConvolution} for $\sigma\ast\sigma$, and appealing to the defining identity \eqref{ImplicitIdentity}, yields
\begin{align}
(\sigma\ast\sigma)(\xi, \tau)
 &=\int_{\R^d}\ddirac{\tau-\psi(\xi/2-Ty)-\psi(\xi/2+Ty)}|\det T'(y)|\d y\notag\\
 &=\int_{\R^d}\ddirac{\tau-2\phi(\xi/2)-\vphi(\xi/2-y)-\vphi(\xi/2+y)}|\det T'(y)|\d y.\label{beforestrict}
\end{align}
From Lemma \ref{contractive}, we know that  $|\det T'|\leq 1$, and so H\"older's inequality implies
\[(\sigma\ast\sigma)(\xi,\tau)
\leq
 (\sigma_0\ast\sigma_0)(\xi,\tau-2\phi(\xi/2)),\]
for every $\xi\in\R^d$ and $\tau>0$. This is equivalent to inequality \eqref{comparisonineq}.
We now use the full power of \eqref{<1} to argue that this inequality must be strict at every point in the interior of the support of $\sigma\ast\sigma$. Let  $(\xi,\tau)$ be one such point, for which $c:=\tau-2\psi(\xi/2)>0$. It is straightforward to check that the singular measure that is being integrated in \eqref{beforestrict} is supported on the ellipsoid $\mathcal E_\vphi(\xi,c)$. Since $c>0$, this ellipsoid does not contain the origin, and by Lemma \ref{contractive}  the strict inequality $|\det T'(y)|<1$ holds at every point $y\in \mathcal E_\vphi(\xi,c)$. This can be strengthened to $|\det T'(y)|\leq c_0$ for some fixed $c_0<1$ (which depends on $\phi,\xi,\tau$ but not on $y$), since the set $\mathcal E_\vphi(\xi,c)$ is compact and the function $y\mapsto\det T'(y)$ is continuous. The result now follows from replacing the $\delta$-function appearing in the integral \eqref{beforestrict} by an appropriate $\varepsilon$-neighborhood of the ellipsoid $\mathcal E_\vphi(\xi,c)$, and then analyzing 
the cases of equality in H\"older's inequality. To conclude the proof of the theorem, let $\varepsilon\to 0^+$.
\end{proof}

\begin{remark}
The previous discussion can be partially generalized to the case of $n$-fold convolutions for $n\geq 3$.
Defining the functions 
$$g_n(t)=\sum_{j=1}^{n-1}\psi(\xi/n-ty_j)+\psi(\xi/n+t\sum_{j=1}^{n-1}y_j)-n\psi(\xi/n),$$
$$h_n(t)=\sum_{j=1}^{n-1}\vphi(\xi/n-ty_j)+\vphi(\xi/n+t\sum_{j=1}^{n-1}y_j)-n\vphi(\xi/n),$$
we have the following generalization of Lemma \ref{convexity}, whose straightforward proof (omitted) can be done by induction on $n$.

\begin{lemma}\label{n-convexity}
Let $n\geq 2$. Let $\psi,\varphi:\R^d\to\R$ be differentiable, convex functions, such that their difference $\psi-\varphi$ is also  convex.
 Given $\xi,y_1,\dotsc,y_{n-1}\in\R^d$, define the functions $g_n, h_n$  as above. Then:
 \begin{enumerate}
 \item[(a)]  $g_n(t)\geq h_n(t)\geq 0$, for every $t\in\R$.
  \item[(b)] The functions $g_n$ and $h_n$ are convex.
  \item[(c)] $g_n'(0)=h_n'(0)=0$.
  \item[(d)] If $\psi$ is strictly convex and $(y_1,\dots,y_{n-1})\neq (0,\ldots,0)$, then $g_n$ attains its unique global minimum at $t=0$.
  \item[(e)] If $\psi$ is strictly convex and $(y_1,\dots,y_{n-1})\neq (0,\ldots,0)$, then there exists a unique nonnegative $\la=\la(y_1,\dotsc,y_{n-1},\xi)$ such that
  \[h_n(1)=g_n(\la),\]
  moreover $\la\in[0,1]$.
 \item[(f)] If $h_n(1)>0$, then $\la>0$. If $h_n(1)<g_n(1)$, then $\la<1$.
 \end{enumerate}
 \end{lemma}

\noindent An $n$-linear version of Theorem \ref{comparison-convolution} would follow from satisfactory substitutes for Lemmata \ref{transformation} and \ref{contractive}. The latter is more intricate if $n\geq 3$, and the authors have not investigated the extent to which the argument would need to be changed.
\end{remark}

\section{Optimal constants and nonexistence of extremizers}\label{sec:NoExt}

This chapter is devoted to the proof of Theorem \ref{main}. 
In what follows,
 the function $\phi:\R^2\to\R$ is assumed to be nonnegative, 
twice continuously differentiable and strictly convex, 
 $\sigma$ denotes projection measure on the surface $\Sigma_\phi\subset\R^3$, and $\psi=|\cdot|^2+\phi$.
 We start by stating two lemmata which explore the connection between pointwise values of the 
 convolution measure $\sigma\ast\sigma$, and concentration at a point.

\begin{lemma}
 \label{upper-bound-concentration}
 Let $y_0\in\R^2$ be given, and let $\{f_n\}\subset L^2(\R^2)$ be a sequence concentrating at $y_0$. Then 
  \begin{equation}
  \label{value_boundary}
  \limsup_{n\to\infty}\frac{\norma{f_n\sigma\ast f_n\sigma}_{L^2(\R^{3})}^2}{\norma{f_n}_{L^2(\R^2)}^4}\leq (\sigma\ast\sigma)(2y_0,2\psi(y_0)).
 \end{equation}
\end{lemma}

\begin{lemma}
 \label{lemma-concentration-point}
 Let $y_0\in\R^2$ be given, and let $f_n(y)=e^{-n(\psi(y)-\psi(y_0)-\langle\nabla\psi(y_0),y-y_0\rangle)}$. Then the sequence 
$\{f_n/\norma{f_n}_{L^2}\}$ concentrates at $y_0$, and
 \begin{equation}
  \label{concentration_value_boundary}
  \lim_{n\to\infty}\frac{\norma{f_n\sigma\ast f_n\sigma}_{L^2(\R^3)}^2}{\norma{f_n}_{L^2(\R^2)}^4}=(\sigma\ast\sigma)(2y_0,2\psi(y_0)).
 \end{equation}
\end{lemma}

\begin{proof}[Proof of Lemma \ref{lemma-concentration-point}]
We first prove that the given sequence concentrates at $y_0$. With that purpose in mind, fix $\rho>0$. The function 
$$\gamma(y):=\psi(y)-\psi(y_0)-\langle\nabla\psi(y_0),y-y_0\rangle$$
 satisfies $\gamma\geq 0$, $\gamma(y_0)=0$, 
$\nabla\gamma(y_0)=0$ and $H(\gamma)(y_0)=2I+H(\phi)(y_0)$. It follows that, for any  sufficiently small $\eps>0$, there exists $r=r_\eps>0$ such that the inequality
\[\gamma(y)\leq (1+\eps)\Big(\ab{y-y_0}^2+\frac{1}{2}\langle 
y-y_0,H(\phi)(y_0)\cdot (y-y_0)\rangle\Big)\]
holds, for every $y\in\R^2$ satisfying 
$\ab{y-y_0}\leq r$. The $L^2$ norm of the function $f_n$ can be bounded from below as follows:
\begin{align*}
\norma{f_n}_{L^2}^2&
=\int_{\R^2}e^{-2n\gamma(y)}\d y
\geq \int_{\{\ab{y-y_0}\leq r\}}e^{-2n(1+\eps)\big(\ab{y-y_0}^2+\frac{1}{2}\langle y-y_0,H(\phi)(y_0)\cdot (y-y_0)\rangle\big)}\d y\\
&=\int_{\{\ab{y}\leq r\}}e^{-2n(1+\eps)\langle y,A\cdot y\rangle}\d y
\geq\frac 1{(\det A)^{\frac12}}\int_{\{\ab{y}\leq \alpha r\}}e^{-2n(1+\eps)\ab{y}^2}\d y\\
&=\frac{2\pi}{(\det A)^{\frac12}}\frac{1-e^{-2n(1+\eps)\alpha^2r^2}}{4n(1+\eps)},
\end{align*}
where the (positive-definite) matrix $A$ is given by $A=I+\frac{1}{2}H(\phi)(y_0)$, and $\alpha>0$ denotes the square root of the smallest eigenvalue of $A$. Noting that 
$$\gamma(y)=\ab{y-y_0}^2+\phi(y)-\phi(y_0)-\langle\nabla\phi(y_0),y-y_0\rangle\geq \ab{y-y_0}^2,$$ 
we obtain 
\begin{align*}
 \int_{\{\ab{y-y_0}\geq \rho\}}\ab{f_n(y)}^2\d y
 \leq\int_{\{\ab{y-y_0}\geq \rho\}} e^{-2n\ab{y-y_0}^2}\d y
 =2\pi\frac{e^{-2n\rho^2}}{4n}.
\end{align*}
Therefore
\begin{align*}
\norma{f_n}_{L^2}^{-2} {\int_{\{\ab{y-y_0}\geq \rho\}}\ab{f_n(y)}^2\d y}\leq 
(1+\eps)(\det A)^{\frac12}\frac{e^{-2n\rho^2}}{1-e^{-2n(1+\eps)\alpha^2r^2}}\to 0,
\end{align*}
as $n\to\infty$, as had to be shown.
We now turn to the proof of identity \eqref{concentration_value_boundary}. Start by noting that the function 
$\gamma$ equals the restriction of the linear affine function 
$$(\xi,\tau)\mapsto\tau-\psi(y_0)-\langle\nabla\psi(y_0),\xi-y_0\rangle$$ 
to the surface $\Sigma_\phi\subset \R^3$.
It follows that
\[(f_n\sigma\ast f_n\sigma)(\xi,\tau)=e^{-n(\tau-\langle\nabla\psi(y_0), 
\xi\rangle)}e^{2n(\psi(y_0)-\langle\nabla\psi(y_0), y_0\rangle)}(\sigma\ast\sigma)(\xi,\tau),\]
which in turn implies the pointwise identity
\begin{equation}\label{convolutionfactors}
(f_n \sigma\ast f_n\sigma)^2=(f_n^2\sigma\ast f_n^2\sigma)(\sigma\ast\sigma).
\end{equation}
Given $r>0$, let 
$$E_r:=\{(y,\psi(y))\in\R^{2+1}|\,y\in B_r(y_0)\}\subset\Sigma_\phi$$
 denote the cap of radius $r$ and center $(y_0,\psi(y_0))$ on the surface $\Sigma_\phi$. From identity \eqref{convolutionfactors}, it follows that
\begin{align*}
 \norma{f_n\sigma\ast f_n\sigma}_{L^2(\R^3)}^2
 &=\int_{\R^3}\bigl(f_n^2\one_{E_r}\sigma\ast f_n^2\one_{E_r}\sigma\bigr)(\xi,\tau) \bigl(\sigma\ast\sigma\bigr)(\xi,\tau)\d\xi \d\tau\\
 &\quad+\int_{\R^3}\bigl(f_n^2\one_{E_r^\complement}\sigma\ast f_n^2\one_{E_r^\complement}\sigma\bigr)(\xi,\tau)\bigl(\sigma\ast\sigma\bigr)(\xi,\tau)\d\xi \d\tau\\
 &\quad\quad+2\int_{\R^3}\bigl(f_n^2\one_{E_r}\sigma\ast f_n^2\one_{E_r^\complement}\sigma\bigr)(\xi,\tau)\bigl(\sigma\ast\sigma\bigr)(\xi,\tau)\d\xi \d\tau,
\end{align*}
where $E_r^\complement$ stands for the complement of the set $E_r$ in $\Sigma_\phi$.
Dividing by $\norma{f_n}_{L^2}^4$, we can bound the last summand by
\begin{multline*}
\norma{f_n}_{L^2}^{-4}\int_{\R^3}\bigl(f_n^2\one_{E_r}\sigma\ast f_n^2\one_{E_r^\complement}\sigma\bigr)(\xi,\tau)\bigl(\sigma\ast\sigma\bigr)(\xi,\tau)\d\xi \d\tau\\
\leq\sup_{(\xi,\tau)\in 
\R^3}(\sigma\ast\sigma)(\xi,\tau)\frac{\norma{f_n\one_{E_r'}}_{L^2}^2}{\norma{f_n}_{L^2}^{2}}\frac{\norma{f_n\one_{E_r'^\complement}}_{L^2}
^2 } {\norma{f_n}_{L^2}^{2}},
\end{multline*}
where
$E_r':=B_r(y_0)\subset 
\R^2$, and $E_r'^\complement$ stands for the complement of the set $E_r'$ in $\R^2$.
The right-hand side of this inequality tends to zero, as $n\to\infty$, because the sequence $\{f_n/\|f_n\|_{L^2}\}$ concentrates at the point $y_0$. The second summand can be treated in an analogous way.
The first summand, after appropriate normalization, is bounded from above by
\begin{multline*}
 \norma{f_n}_{L^2}^{-4}\int_{\R^3}
 \bigl(f_n^2\one_{E_r}\sigma\ast f_n^2\one_{E_r}\sigma\bigr)(\xi,\tau)
 \bigl(\sigma\ast\sigma\bigr)(\xi,\tau)\d\xi \d\tau\\
 \leq 
\frac{\norma{f_n\one_{E_r'}}_{L^2}^4}{\norma{f_n}_{L^2}^{4}}
\sup_{(\xi,\tau)\in E_r+E_r}(\sigma\ast\sigma)(\xi,\tau) ,
\end{multline*}
and from below by
\begin{multline*}
\norma{f_n}_{L^2}^{-4}\int_{\R^3}
\bigl(f_n^2\one_{E_r}\sigma\ast f_n^2\one_{E_r}\sigma\bigr)(\xi,\tau)
\bigl(\sigma\ast\sigma\bigr)(\xi,\tau)\d\xi \d\tau\\
\geq \frac{\norma{f_n\one_{E_r'}}_{L^2}^4}{\norma{f_n}_{L^2}^{4}}
\inf_{(\xi,\tau)\in E_r+E_r}(\sigma\ast\sigma)(\xi,\tau).
\end{multline*}
Since $\norma{f_n\one_{E_r'}}_{L^2}/\norma{f_n}_{L^2}\to 1$, as $n\to\infty$, we obtain
\[\limsup_{n\to\infty}\frac{\norma{f_n\sigma\ast f_n\sigma}_{L^2(\R^3)}^2}{\norma{f_n}_{L^2(\R^2)}^4}\leq \sup_{(\xi,\tau)\in 
E_r+E_r}(\sigma\ast\sigma)(\xi,\tau),\]
and
\[\liminf_{n\to\infty}\frac{\norma{f_n\sigma\ast f_n\sigma}_{L^2(\R^3)}^2}{\norma{f_n}_{L^2(\R^2)}^4}\geq \inf_{(\xi,\tau)\in 
E_r+E_r}(\sigma\ast\sigma)(\xi,\tau).\]
Identity \eqref{concentration_value_boundary}  follows because the convolution $\sigma\ast\sigma$ defines a continuous function up to the boundary of its support, and $r>0$ was 
arbitrary. Taking $r\to 0^+$ finishes the proof.
\end{proof}

\begin{proof}[Sketch of proof of Lemma \ref{upper-bound-concentration}] 
Integrate the pointwise bound
$$\ab{(f\sigma\ast f\sigma)(\xi,\tau)}^2\leq 
\bigl(\ab{f}^2\sigma\ast\ab{f}^2\sigma\bigr)(\xi,\tau) \bigl(\sigma\ast\sigma\bigr)(\xi,\tau),$$
which was observed in \cite{F,OS,Q} to hold almost everywhere, and proceed as in the proof of the corresponding inequality in Lemma \ref{lemma-concentration-point}.
\end{proof}

\begin{proof}[Proof of Theorem \ref{main}]
As in the proof of Lemma \ref{upper-bound-concentration}, the Cauchy--Schwarz inequality implies 
\begin{equation}\label{firstCS}
\norma{f\sigma\ast f\sigma}^2_{L^2(\R^{3})}\leq \norma{\sigma\ast\sigma}_{L^\infty(\R^3)}\norma{f}_{L^2(\R^2)}^4.
\end{equation}
 It follows that (a possibly non-sharp version of) inequality \eqref{2-convolution-2} holds, as long as the $L^\infty$ norm of the convolution  $\sigma\ast\sigma$ is finite. This, in turn, can be seen using  
identity \eqref{ConvolutionFormula}, since the Hessian of $\psi$ satisfies $H(\psi)=2I+H(\phi)$, and the matrix $H(\phi)(x)$ is positive semidefinite, for every $x\in\R^2$.  
Estimate \eqref{firstCS} also shows that the optimal constant in inequality \eqref{2-convolution-2} satisfies 
$$\mathcal R^4_\phi\leq \norma{\sigma\ast\sigma}_{L^\infty}.$$
Now, let $\sigma_0$ denote the projection measure on the paraboloid $\Sigma_0$.
 From Theorem \ref{comparison-convolution} and Remark \ref{Foschi}, we know that $\norma{\sigma\ast\sigma}_{L^\infty}\leq\norma{\sigma_0\ast\sigma_0}_{L^\infty}=\frac{\pi}2$.  That these two quantities are actually the same follows from the fact that the convolution $\sigma\ast\sigma$ attains the value $\pi/2$  at the boundary point $(2y_0,2\psi(y_0))$ in case (i), or at infinity in case (ii).
 Identity \eqref{ValueS} will then follow from the inequality
 \begin{equation}\label{Sgeq}
 \mathcal R^4_\phi\geq \frac \pi 2,
 \end{equation}
 which we establish using the sequences given by \eqref{extseq}. We consider the two cases separately. In case (i), since $H(\phi)(y_0)=0$, it follows from identity \eqref{bdryvaluesDet} that \mbox{$(\sigma\ast\sigma)(2y_0,2\psi(y_0))=\frac{\pi}2$}, and therefore the sequence $\{f_n/\|f_n\|_{L^2}\}$, where
 $$f_n(y)=e^{-n(\psi(y)-\psi(y_0)-\langle\nabla\psi(y_0),y-y_0\rangle)},$$
 is extremizing for inequality \eqref{2-convolution-2} in light of Lemma \ref{lemma-concentration-point}. In case (ii), we have that 
$(\sigma\ast\sigma)(2y_n,2\psi(y_n))\to\frac{\pi}2$, as $n\to\infty$. Choose a sequence $\{a_n\}\subset\N$ in such a way that, for every $n\in\N$, the function given by 
$$f_n(y)=e^{-a_n(\psi(y)-\psi(y_n)-\langle\nabla\psi(y_n),y-y_n\rangle)}$$
satisfies
 \[\left\vert{\frac{\norma{f_{n}\sigma\ast 
f_{n}\sigma}_{L^2}^2}{\norma{f_{n}}_{L^2}^4}-(\sigma\ast\sigma)(2y_n,2\psi(y_n))}\right\vert\leq\frac{1}{n},\]
and
\begin{equation}\label{conctrinfty}
\int_{\{|y-y_n|\geq \frac1n\}} |f_{n}(y)|^2 \d y \leq \frac{1}{n} \|f_{n}\|_{L^2}^2.
\end{equation}
That this is possible follows again from Lemma \ref{lemma-concentration-point}. Since 
$$\frac{\norma{f_{n}\sigma\ast 
f_{n}\sigma}_{L^2}^2}{\norma{f_{n}}_{L^2}^4}\to\frac\pi 2, \textrm{ as } n\to\infty,$$
the sequence $\{f_n/\|f_n\|_{L^2}\}$ is again extremizing for inequality  \eqref{2-convolution-2}.  
 This establishes \eqref{Sgeq} in both cases (i) and (ii), and therefore identity \eqref{ValueS} is proved.
 Incidentally, note that condition \eqref{conctrinfty} ensures that  $\{f_n/\|f_n\|_{L^2}\}$
concentrates along the sequence $\{y_n\}$. Since $|y_n|\to\infty$, as $n\to\infty$, it concentrates at infinity.
 
\noindent We finish by showing that extremizers for inequality \eqref{2-convolution-2} do not exist. Aiming at a contradiction, let $f$ be an extremizer. An application of Cauchy--Schwarz and H\"older's inequalities yields 
\begin{align*}
{\mathcal R_\phi^4} \norma{f}_{L^2}^4
&=\norma{f\sigma\ast f\sigma}_{L^2}^2\\
&\leq \int_{\R^3} |(f^2\sigma\ast f^2\sigma)(\xi,\tau)|(\sigma\ast \sigma)(\xi,\tau)\d\xi \d\tau\\
&\leq \|\sigma\ast\sigma\|_{L^\infty}\int_{\R^3} |(f^2\sigma\ast f^2\sigma)(\xi,\tau)|\d\xi \d\tau\\
&= \|\sigma\ast\sigma\|_{L^\infty} \|f\|_{L^2}^4.
 \end{align*}
 Since $\mathcal R_\phi^4=\|\sigma\ast\sigma\|_{L^\infty}=\frac{\pi}2$ and $f\neq 0$, all inequalities in this chain of inequalities must be equalities. In 
particular, the convolution $\sigma\ast\sigma$ must be constant equal to $\|\sigma\ast\sigma\|_{L^\infty}$ almost everywhere inside the support of 
 $f^2\sigma\ast f^2\sigma$, which is a set of positive Lebesgue measure since $f\neq 0$. This contradicts the strict inequality 
 \begin{equation*}
(\sigma\ast\sigma)(\xi, \tau)<\norma{\sigma\ast\sigma}_{L^\infty},\textrm{ for almost every }(\xi, \tau)\in\textrm{supp} (\sigma\ast\sigma),
\end{equation*}
 which in turn is an immediate consequence of the second part of Theorem \ref{comparison-convolution}. This contradiction shows that extremizers do not exist. The proof of the theorem is now complete.
\end{proof}

\section{On extremizing sequences}
\label{sec:BehaviorExtSeq}

From the previous chapter, we know that extremizers for inequality \eqref{2-convolution-2}  do not exist. As mentioned in the Introduction, this failure of compactness can be understood via the concentration-compactness principle, which is the subject of the present chapter.
Heuristically, an extremizing sequence  for inequality \eqref{2-convolution-2}  should concentrate around the points where the function $\sigma\ast\sigma$ 
achieves its essential supremum.
Lemma \ref{upper-bound-concentration} and formula \eqref{bdryvaluesDet} imply that, if an extremizing sequence concentrates at a point 
$y_0$, then necessarily $H(\phi)(y_0)=0$. Lemma \ref{lemma-concentration-point} provides the construction of  an explicit extremizing sequence 
concentrating at any point $y_0\in\R^2$, provided $H(\phi)(y_0)=0$. Therefore, concentration occurs at a point if and only if the Hessian vanishes at that point. 
Further information concerning extremizing sequences concentrating at spatial infinity will be obtained below.

\subsection{Weak interaction between distant caps}\label{sec:widc}

Reasoning in a similar way to the proof of \eqref{ConvolutionFormula} from Proposition \ref{ConvolFormulaProp}, we find that the identity
\begin{equation}\label{ConvFandG}
(f\sigma\ast g\sigma)(\xi,\tau)
=
\int_{\mathbb{S}^1}\frac{f(\xi/2+\alpha(\xi,\tau, \omega)\omega)g(\xi/2-\alpha(\xi,\tau, \omega)\omega)}{
\int_{-1}^1
\langle \omega, 
H(\psi)(\xi/ 2+t\alpha(\xi,\tau, \omega)\omega)
\cdot \omega\rangle
\d t}
\d\mu_\omega
\end{equation}
holds, in particular, in the case when $f,g$ are indicator functions of balls or their complements.
Formula \eqref{ConvFandG} allows for a quantification of the general principle that ``distant caps interact weakly''. This is a geometric feature that translates into useful bilinear estimates, and has been observed in a variety of related contexts; see, for instance, \cite{CS, OS2}. The precise  statement is as follows.

\begin{lemma}
 \label{WeakInteraction}
 Let $r,\rho>0$ satisfy $\rho>3r$. Then, for any $y_0\in\R^2$,
 \begin{equation}\label{weakinterac}
 \norma{\one_{B_r(y_0)}\sigma\ast \one_{B^\complement_\rho(y_0)}\sigma}_{L^\infty(\R^3)}\leq \frac 12\arcsin\Bigl(\frac{2r}{\rho-r}\Bigr).
 \end{equation}
 As a result, the following statements hold:
 \begin{itemize}
  \item[(a)] \label{w-1} For any  $r>0$ and $y_0\in\R^2$, 
 \begin{equation}
  \label{weak_interaction_balls}
  \lim_{\rho\to\infty}\norma{\one_{B_r(y_0)}\sigma\ast \one_{B^\complement_\rho(y_0)}\sigma}_{L^\infty(\R^3)}= 0.
 \end{equation}
 \item[(b)]\label{w-2} For any $\rho>0$ and $y_0\in\R^2$, 
 \[\lim_{r\to0^+}\norma{\one_{B_r(y_0)}\sigma\ast \one_{B^\complement_\rho(y_0)}\sigma}_{L^\infty(\R^3)}= 0.\]
 \item[(c)] \label{w-3} 
  For any $r>0$, 
  \begin{equation}
  \label{weak_interaction_balls_v2}
  \lim_{\rho\to\infty}\sup_{y\in\R^2}\norma{\one_{B_r(y)}\sigma\ast \one_{B^\complement_\rho(y)}\sigma}_{L^\infty(\R^3)}= 0.
 \end{equation}
\end{itemize}
  Moreover,
    \begin{itemize}
  \item[(d)] If $B,B'\subseteq\R^2$ are disjoint balls, then
 \[\norma{\one_{B}\sigma\ast \one_{B'}\sigma}_{L^\infty(\R^3)}\leq \frac{\pi}{4}.\]
 \end{itemize}
\end{lemma}

\begin{proof}
We establish identity \eqref{weakinterac} for $y_0=0$ only, the case of general $y_0\in\R^2$ being similar. Let $\rho>r>0$. If $f=\one_{B_r}$ 
and $g=\one_{B^\complement_\rho}$, then the integrand in 
\eqref{ConvFandG} is nonzero only if the point $(\xi,\tau)$ satisfies
\[ \xi/2+\alpha(\xi,\tau, \omega)\omega\in B_r,\,\text{ and }\, \xi/2-\alpha(\xi,\tau, \omega)\omega\notin B_\rho.\]
By the triangle inequality, this can only happen if
\begin{equation}
 \label{LowerBoundXi}
 \ab{\xi}\geq \ab{\xi/2-\alpha(\xi,\tau, \omega)\omega}-\ab{\xi/2+\alpha(\xi,\tau, \omega)\omega}\geq \rho-r.
\end{equation}
In this case, if $\rho>3r$, then  $\ab{\xi/2}>r$, and therefore the ray $\xi/2+t\omega$, $t>0$, intersects the ball $B_r$ only if 
$\omega$ belongs to an arc of $\mathbb S^1$ 
of measure 
exactly $2\arcsin(\frac{2r}{\ab{\xi}})$.
Denoting  arc length measure on the unit circle by $\mu$ as usual, we conclude that 
\begin{multline*}
\mu(\{\omega\in \mathbb S^1: \xi/2+\alpha(\xi,\tau, \omega)\omega\in B_r,\,\xi/2-\alpha(\xi,\tau, \omega)\omega\notin B_\rho\})\\
\leq 2\arcsin\Big(\frac{2r}{\ab{\xi}}\Big)
\leq 2\arcsin\Big(\frac{2r}{\rho-r}\Big).
\end{multline*}
 It follows that, for every $(\xi,\tau)\in\R^3$,
\[(\one_{B_r}\sigma\ast \one_{B^\complement_\rho}\sigma)(\xi,\tau)\leq \frac{1}{2}\arcsin\Big(\frac{2r}{\rho-r}\Big),\]
where we bounded the denominator in \eqref{ConvFandG} from below by $4$. Parts (a) and (b) follow at once, and a similar reasoning for $y_0\neq 0$ establishes (c).
For part (d), note that the definition \eqref{implicitlambda} of the function $\la$ implies  $\la(-w,\xi)=\la(w,\xi)$ for every $w,\xi$, and 
therefore the function $\alpha$ satisfies $\alpha(\xi,\tau,-\omega)=\alpha(\xi,\tau,\omega)$, for every \mbox{$\omega\in\mathbb S^1$}. It then follows that, if $\xi/2+\alpha(\xi,\tau, 
\omega)\omega\in B$ and $\xi/2-\alpha(\xi,\tau, \omega)\omega\in B'$, then \mbox{$\xi/2+\alpha(\xi,\tau,-\omega)(-\omega)\notin B$} and $\xi/2-\alpha(\xi,\tau,-\omega)(-\omega)\notin B'$. As a consequence, the subset of $\mathbb S^1$ where the integrand  
in \eqref{ConvFandG} is nonzero has measure bounded from above by $\pi$, and the result follows as before.
\end{proof}

\subsection{Concentration-compactness} 
The three lemmata in this section hold under the general hypotheses of Theorem \ref{main}, which for brevity will not be included in the corresponding statements. 

\begin{lemma}
 \label{improvement-implies-concentration}
 Under the hypotheses of Theorem \ref{main}, suppose that there exist a subset $X\subset(\R^2)^2$ and $\delta>0$  such that, for every $(y,z)\in X$,
 \begin{equation}
  \label{improved-estimate}
  \frac{(\sigma\ast\sigma)(y+z,\psi(y)+\psi(z))}{\|\sigma\ast \sigma\|_{L^\infty(\R^3)}}\leq 1-\delta.
 \end{equation}
 Let $\{f_n\}\subset L^2(\R^2)$ be any extremizing sequence for inequality \eqref{2-convolution-2}. Then
 \[\int_X \ab{f_n(y)}^2\ab{f_n(z)}^2\d y\d z\to0,\text{ as }n\to\infty.\]
 In particular, if $X$ contains a subset of the form $A\times B$, for some $A,B\subset\R^2$, then
 \[\int_A \ab{f_n(y)}^2\d y\int_B \ab{f_n(z)}^2\d z\to0\text{, as }n\to\infty.\]
\end{lemma}
\begin{proof}
Let $\{f_n\}\subset L^2(\R^2)$ be an extremizing sequence for inequality \eqref{2-convolution-2}. The first step is to verify that 
\begin{equation}\label{HypProp}
\liminf_{n\to\infty}\int_{(\R^2)^2} |f_n(y)|^2 |f_n(z)|^2
\frac{(\sigma\ast\sigma)(y+z,\psi(y)+\psi(z))}{\|\sigma\ast \sigma\|_{L^\infty}} \d y\d z=1.
\end{equation}
With this goal in mind, estimate
\begin{align*}
\int_{\R^3}\ab{(f_n\sigma\ast f_n\sigma)(\xi,\tau)}^2\d\xi \d\tau 
&\leq \int_{\R^3}\bigl(\ab{f_n}^2\sigma\ast\ab{f_n}^2\sigma\bigr)(\xi,\tau) 
\bigl(\sigma\ast\sigma\bigr)(\xi,\tau)\d\xi \d\tau \\
&=\int_{(\R^2)^2} |f_n(y)|^2 |f_n(z)|^2
(\sigma\ast\sigma)(y+z,\psi(y)+\psi(z)) \d y\d z\\
&\leq \|\sigma\ast \sigma\|_{L^\infty}\norma{f_n}_{L^2}^4.
\end{align*}
The first and the last terms in this chain of inequalities converge to $\|\sigma\ast \sigma\|_{L^\infty}$, as 
$n\to\infty$,  and therefore so does the third term, and  \eqref{HypProp} follows.
We next observe 
\[\liminf_{n\to\infty}\int_{(\R^2)^2} |f_n(y)|^2 |f_n(z)|^2 \d y\d z=\lim_{n\to\infty}\norma{f_n}_{L^2}^4=1.\]
Writing $X^\complement$ for the complement of the set $X$ in $(\R^2)^2$, we have an inequality
\begin{align*}
 \int_{(\R^2)^2} |f_n(y)|^2 |f_n(z)&|^2
\frac{(\sigma\ast\sigma)(y+z,\psi(y)+\psi(z))}{\|\sigma\ast \sigma\|_{L^\infty}} \d y\d z&\\
&\leq (1-\delta)\int_{X} |f_n(y)|^2 |f_n(z)|^2 \d y\d z+\int_{X^\complement} |f_n(y)|^2 |f_n(z)|^2 \d y\d z.
\end{align*}
Since $\norma{f_n}_{L^2}\to 1$ as $n\to\infty$, we conclude from \eqref{HypProp} that 
\begin{align*}
 1&\leq \liminf_{n\to\infty}\Bigl(\Bigl(\int_{\R^2}\ab{f_n(y)}^2\d y\Bigr)^2-\delta\int_{X}\ab{f_n(y)}^2\ab{f_n(z)}^2\d y\d z\Bigr)\\
 &=1-\delta\limsup_{n\to\infty}\int_{X}\ab{f_n(y)}^2\ab{f_n(z)}^2 \d y\d z.
\end{align*}
It follows that
\begin{equation*}
 \limsup_{n\to\infty}\int_{X}\ab{f_n(y)}^2\ab{f_n(z)}^2 \d y\d z=0,
\end{equation*}
which establishes the first statement. The second statement follows at once, and the proof is complete.
\end{proof}

\noindent The preceding lemma implies the following modest amount of control over extremizing sequences that split their mass in a nontrivial way.

\begin{lemma}\label{NoSplitting}
 Under the hypotheses of Theorem \ref{main}, let $\{f_n\}\subset L^2(\R^2)$ be any extremizing sequence for inequality \eqref{2-convolution-2}. 
Let $0<r_1<r_2<r_3<\infty$ be arbitrary. Then
\[\int_{B_{r_1}}\ab{f_n(y)}^2 \d y\int_{B_{r_3}\setminus B_{r_2}}\ab{f_n(z)}^2 \d z\to 0, \textrm{ as } n\to\infty.\]
\end{lemma}

\begin{proof}
 Let $X=B_{r_1}\times (B_{r_3}\setminus B_{r_2})$. Appealing to the continuity of the convolution $\sigma\ast\sigma$ on its support, to the fact that the essential 
supremum is only achieved on the boundary of the support (as observed in the course of the proof of Theorem \ref{comparison-convolution}), together with the compactness of the set $\overline X$ and the fact that $r_1<r_2$, we can ensure the existence of  $\delta=\delta_{r_1,r_2,r_3}>0$ such that
\[ \frac{(\sigma\ast\sigma)(y+z,\psi(y)+\psi(z))}{\|\sigma\ast \sigma\|_{L^\infty}}\leq 1-\delta,\]
for every $(y,z)\in X$. The conclusion now follows from Lemma \ref{improvement-implies-concentration}.
\end{proof}

\noindent Lemma \ref{NoSplitting} can be upgraded in a way that reveals that an extremizing sequence can only split its mass in a nontrivial way if neither of the corresponding supports remains in a bounded region. We formulate one version of this principle which will be useful for our purposes.

\begin{lemma}\label{NoSplitting_part2}
 Under the hypotheses of Theorem \ref{main}, let $\{f_n\}\subset L^2(\R^2)$ be any extremizing sequence for inequality \eqref{2-convolution-2}. Let $0<r_1<r_2<\infty$ be arbitrary. Then
\begin{equation}
 \label{ProdToZero_2}
 \int_{B_{r_1}}\ab{f_n(y)}^2 \d y\int_{\R^2\setminus B_{r_2}}\ab{f_n(z)}^2 \d z\to 0, \textrm{ as }n\to\infty.
\end{equation}
\end{lemma}

\begin{remark}\label{rtorho}
If conclusion \eqref{ProdToZero_2} holds for one pair $(r_1,r_2)$ satisfying $0<r_1<r_2<\infty$, then it holds for any pair $(\rho_1,\rho_2)$ satisfying $0<\rho_1<\rho_2<\infty$ and $r_1\geq\rho_1$. To see this, start by 
noticing that the case $r_1\geq\rho_1$ and $r_2\leq\rho_2$ is clear. On the other hand, if $r_1\geq\rho_1$ and $r_2>\rho_2$, then
\begin{align*}
 \int_{B_{\rho_1}}&\ab{f_n(y)}^2 \d y\int_{\R^2\setminus B_{\rho_2}}\ab{f_n(z)}^2 \d z\\
&\leq \int_{B_{r_1}}\ab{f_n(y)}^2 \d y\int_{\R^2\setminus 
B_{r_2}}\ab{f_n(z)}^2 \d z
\quad+\int_{B_{\rho_1}}\ab{f_n(y)}^2 \d y\int_{B_{r_2}\setminus 
B_{\rho_2}}\ab{f_n(z)}^2 \d z,\\
\end{align*}
which tends to zero, as $n\to\infty$, by \eqref{ProdToZero_2} and Lemma \ref{NoSplitting}. 
Moreover, in view of the uniform bound with respect to $y\in\R^2$ from part (c) of Lemma \ref{WeakInteraction},
the proof 
given below can be adapted to the case of balls centered at any point $y\in\R^2$, not necessarily the origin. 
\end{remark}

\begin{proof}[Proof of Lemma \ref{NoSplitting_part2}]
Let $r_1>0$ be given.
 If 
 $$\int_{B_{r_1}}\ab{f_n(y)}^2 \d y\to0, \textrm{ as } n\to\infty,$$ 
 then the conclusion follows at once since $\norma{f_n}_{L^2}\leq 1$. Therefore no generality is lost in assuming, 
possibly after passing to a subsequence, that
\begin{equation}\label{infremainspositive}
\delta:=\inf_{n\in\N}\int_{B_{r_1}}\ab{f_n(y)}^2 \d y>0.
\end{equation}
It suffices to show that
\begin{equation}
 \label{GoesToZero}
 \int_{\R^2\setminus B_{r_2}}\ab{f_n(z)}^2 \d z\to 0,\text{ as }n\to\infty.
\end{equation}

\noindent Take $r_3>r_2$. From Lemma \ref{NoSplitting} and inequality \eqref{infremainspositive}, we know that
\begin{equation}
 \label{InterGoToZero}
 \int_{B_{r_3}\setminus B_{r_2}}\ab{f_n(z)}^2 \d z\to 0,\text{ as }n\to\infty.
\end{equation}
\noindent Decompose 
$$f_n=f_n\one_{B_{r_2}}+f_{n}\one_{\R^2\setminus B_{r_3}}+f_n\one_{B_{r_3}\setminus B_{r_2}}=:F_n+G_n+H_n,$$ 
and note that
\[f_n\sigma\ast f_n\sigma=F_n\sigma\ast F_n\sigma+G_n\sigma\ast G_n\sigma+2 F_n\sigma\ast G_n\sigma+R_n,\]
where, in view of inequality \eqref{2-convolution-2} and estimate \eqref{InterGoToZero}, the remainder term $R_n$ satisfies
$$\norma{R_n}_{L^2(\R^3)}\leq C\norma{H_n}_{L^2(\R^2)}\to 0, \textrm{ as } n\to\infty.$$
The key step is to bound the quantity $\norma{F_n\sigma\ast G_n\sigma}_{L^2}^2$. We have the pointwise inequality 
\[\ab{(F_n\sigma\ast 
G_n\sigma)(\xi,\tau)}^2
\leq \bigl(\ab{F_n}^2\sigma\ast \ab{G_n}^2\sigma\bigr)(\xi,\tau) \bigl(\one_{B_{r_2}}\sigma\ast\one_{\R^2\setminus 
B_{r_3}}\sigma\bigr)(\xi,\tau),\]
which follows from an application of the Cauchy--Schwarz inequality as before. As a consequence, 
\begin{align*}
 \norma{F_n\sigma\ast G_n\sigma}_{L^2(\R^3)}^2
\leq\rho^2(r_2,r_3)\norma{F_n}_{L^2(\R^2)}^2\norma{G_n}_{L^2(\R^2)}^2,
\end{align*}
where the function $\rho$ is given by 
$$\rho(r_2,r_3):=\norma{\one_{B_{r_2}}\sigma\ast\one_{\R^2\setminus 
B_{r_3}}\sigma}_{L^\infty(\R^3)}^{\frac12}.$$
For large values of $r_3$, the sets $B_{r_2}$ and $\R^2\setminus B_{r_3}$ interact weakly as discussed in \S \ref{sec:widc}. Part (a) 
of Lemma \ref{WeakInteraction} implies that, for each fixed $r_2>0$, 
$$\rho(r_2,r_3)\to 0, \textrm{ as } r_3\to\infty.$$
 
\noindent Two applications of Plancherel's Theorem, together with the triangle inequality, imply the following bound for the inner product:
\[\ab{\langle F_n\sigma\ast F_n\sigma,G_n\sigma\ast G_n\sigma\rangle_{L^2}}\leq \norma{F_n\sigma\ast G_n\sigma}_{L^2}^2.\]

\noindent It follows that there exists an absolute constant $C<\infty$, which can be explicitly computed but whose exact numerical value is unimportant for our purposes, for which
\begin{align*}
 \norma{f_n\sigma\ast f_n\sigma}_{L^2}^2&\leq \norma{F_n\sigma\ast F_n\sigma}_{L^2}^2+\norma{G_n\sigma\ast 
G_n\sigma}_{L^2}^2+C\rho(r_2,r_3)+o_n(1)\\
 &\leq \frac{\pi}{2}(\norma{F_n}_{L^2}^4+\norma{G_n}_{L^2}^4)+C\rho(r_2,r_3)+o_n(1)\\
 &=\frac{\pi}{2}\norma{f_n}_{L^2}^4-\pi\norma{F_n}_{L^2}^2\norma{G_n}_{L^2}^2+C\rho(r_2,r_3)+o_n(1),
\end{align*}
Here, we used the sharp inequality \eqref{2-convolution-2}, and orthogonality considerations. 
The function $o_n(1)$ may depend on $r_3$, but satisfies $o_n(1)\to 0$, as $n\to\infty$, for each fixed $r_3$, and is allowed to change 
from line to line. 
Taking $n\to\infty$ in the previous chain of inequalities, we conclude
\[\limsup_{n\to\infty}\norma{F_n}_{L^2}^2\norma{G_n}_{L^2}^2\leq C\rho(r_2,r_3).\]
Consequently,
\[\limsup_{n\to\infty}\norma{G_n}_{L^2}^2\leq \frac{C}{\delta}\rho(r_2,r_3),\]
where $\delta$ was defined in \eqref{infremainspositive}, and therefore,
\[\limsup_{n\to\infty}\bigl(\norma{G_n}_{L^2}^2+\norma{H_n}_{L^2}^2\bigr)\leq \frac{C}{\delta}\rho(r_2,r_3),\]
which is equivalent to
\[\limsup_{n\to\infty}\int_{\R^2\setminus B_{r_2}}\ab{f_n(z)}^2 \d z\leq \frac{C}{\delta}\rho(r_2,r_3).\]
Since the left-hand side of this inequality is independent of $r_3$, and the right-hand side tends to $0$ as $r_3\to 0$, conclusion \eqref{GoesToZero} must hold. The proof of the lemma is now complete.
\end{proof}

\noindent We have  collected all the ingredients necessary to the proof of Theorem \ref{BehaviorExtSeq}.

\begin{proof}[Proof of Theorem \ref{BehaviorExtSeq}]
 Let $\{f_n\}\subset L^2(\R^2)$ be any extremizing sequence for inequality \eqref{2-convolution-2}. Take any subsequence, and slightly abuse notation by again calling it $\{f_n\}$. If this subsequence $\{f_n\}$ does not concentrate at 
infinity, then there exists a further sub-subsequence, still 
denoted by $\{f_n\}$, and a number $r_0<\infty$, such that
\[\inf_{n\in\N}\int_{B_{r_0}}\ab{f_n(y)}^2 \d y>0.\]
From Lemma \ref{NoSplitting_part2}, we conclude
\[\int_{\R^2\setminus B_{2r_0}}\ab{f_n(y)}^2 \d y\to 0,\text{ as }n\to\infty.\]
It follows that
\begin{equation}\label{intheball}
\int_{B_{2r_0}}\ab{f_n(y)}^2 \d y\to 1,\text{ as }n\to\infty.
\end{equation}
As a consequence of Lemma \ref{improvement-implies-concentration}, 
$\norma{\one_{B_{2r_0}}\sigma\ast\one_{B_{2r_0}}\sigma}_{L^\infty}=\norma{\sigma\ast\sigma}_{L^\infty}$, and the supremum is achieved inside the ball $\bar B_{2r_0}$.
In particular, case (i) holds, and  the set 
$$E:=\{y\in \R^2 | H(\phi)(y)=0\}$$
is nonempty.
For $\eps\in (0,1)$, let $N_\eps(E)$ denote the open $\eps$-neighborhood of $E$, and consider the set
$$Y:=\{(y,z)\in B_{3r_0}\times B_{3r_0}| y\in N_\eps(E), z\in N_\eps(E), |y-z|<\eps\}.$$
Let $X:=(\bar B_{3r_0}\times \bar B_{3r_0})\setminus Y$.
We claim that there 
exists $\delta=\delta(\eps)>0$, such that 
\begin{equation}\label{upper-bound-strong}
 \frac{(\sigma\ast\sigma)(y+z,\psi(y)+\psi(z))}{\|\sigma\ast \sigma\|_{L^\infty}}\leq 1-\delta,
\end{equation}
for every $(y,z)\in X$. This follows from the compactness of the set $X$, together with the fact that, if $(y,z)\in X$, then the point $(y+z,\psi(y)+\psi(z))$ is away from the portion of the boundary of the support where the convolution $\sigma\ast\sigma$ attains its essential supremum  in a quantifiable way that depends only on $\eps$.
Since inequality \eqref{upper-bound-strong} holds for every $(y,z)\in X$,  Lemma \ref{improvement-implies-concentration} implies
\[\int_{X}\ab{f_n(y)}^2\ab{f_n(z)}^2 \d y\d z\to 0,\text{ as }n\to\infty.\]
In light of \eqref{intheball}, it then follows that
\begin{equation}\label{goesto1onY}
\int_{Y}\ab{f_n(y)}^2\ab{f_n(z)}^2 \d y\d z\to 1,\text{ as }n\to\infty.
\end{equation}
The remaining of the proof coincides with the second part of the proof of \cite[Proposition 6.3]{Q}, with minor modifications only. We include it for the convenience of 
the reader. 
We seek to locate a sequence $\{y_n\}\subset\R^2$ along which concentration occurs. Fubini's Theorem, and the fact that $\norma{f_n}_{L^2}\leq 1$, together imply
\begin{align*}
 \int_{Y}\ab{f_n(y)}^2\ab{f_n(z)}^2 \d y\d z&=
    \int_{B_{3r_0}\cap N_\eps(E)}\ab{f_n(y)}^2\Bigl(\int_{B_{3r_0}\cap N_\eps(E)\cap B_\eps(y)}\ab{f_n(z)}^2\d z\Bigr)\d y\\
    &\leq\norma{f_n\one_{B_{3r_0}\cap N_\eps(E)}}_{L^2}^2\sup_{y\in B_{3r_0}\cap N_\eps(E)}\int_{B_{3r_0}\cap 
N_\eps(E)\cap B_\eps(y)}\ab{f_n(z)}^2\d z\leq 1.
\end{align*}
From \eqref{goesto1onY}, it then follows that
\[\lim_{n\to\infty}\sup_{y\in B_{3r_0}\cap N_\eps(E)}\int_{B_{3r_0}\cap N_\eps(E)\cap B_\eps(y)}\ab{f_n(z)}^2\d z=1.\]
This implies the existence of a function $N:(0,1)\to\N$, such that
\[\sup_{y\in B_{3r_0}\cap N_\eps(E)}\int_{\{\ab{z-y}\leq\eps\}}\ab{f_n(z)}^2\d z\geq 1-\frac{\eps}{2}, \textrm{ for every } n\geq N(\eps).\]
 Hence, there exists a sequence $\{y_n^\eps\}_{n\geq N(\eps)}\subset \bar{B}_{3r_0}\cap E$, such that
\[\int_{\{\ab{z-y_n^\eps}\leq2\eps\}}\ab{f_n(z)}^2\d z\geq 1-\eps, \textrm{ for every }  n\geq N(\eps). \]
Here, we exchanged the neighborhood $N_\eps(E)$ for the set $E$, at the expense of an extra $\eps$ in the domain of integration.
We proceed to construct the sequence $\{y_n\}$ via a diagonal process. Take $\eps_k=\frac1{k+2}$. We obtain a strictly increasing sequence 
$$N_k:=\max\{N(\eps_j)|\,1\leq j\leq k\}+k,$$ 
and sequences $\{y_n^k\}_{n\geq N_k}$, satisfying
\[\int_{\{\ab{z-y_n^k}\leq\frac{2}{k}\}}\ab{f_n(z)}^2\d z\geq 1-\frac{1}{k},\]
for every $k\geq 1$ and  $n\geq N_k$. 
For each $n\geq N_1$, let $\ell_n:=\sup\{k\in\N|\,N_k\leq n\}$. This is a finite number since the sequence $\{N_k\}$ is strictly increasing. Further note that $n\geq N_{\ell_n}$. Define
\[y_n:=\begin{cases}
       y_n^{\ell_n},&\text{ if }n\geq N_1,\\
       y_0,&\text{ if }n<N_1,
      \end{cases}
\]
where $y_0\in E$ is arbitrary, but fixed. It is then clear that
\[\int_{\{\ab{z-y_n}\leq \frac{2}{\ell_n}\}}\ab{f_n(z)}^2\d z\geq 1-\frac{1}{\ell_n},\]
for every $n\geq N_1$, which implies that $\{f_n\}$ concentrates along the sequence $\{y_n\}$ since $\ell_n\to\infty$, as $n\to\infty$.
The statement about subsequences of $\{f_n\}$ follows by compactness of  the set $E\cap \bar B_{3r_0}$, since every subsequence of $\{y_n\}$ has a further 
sub-subsequence that converges to a point in $E\cap \bar B_{3r_0}$.
\end{proof}

\subsection{Some consequences}
The methods of the proof of Theorem \ref{BehaviorExtSeq} specialize to at least two distinct situations of interest. The first one is a direct consequence of the statement of Theorem \ref{BehaviorExtSeq}. 

\begin{cor}\label{GoingtoInfinity}
 Let $\phi:\R^2\to\R$ be a nonnegative, twice continuously differentiable, strictly convex function, such that
 \begin{itemize}
 \item [(i)] $H(\phi)(y)\neq0$, for every $y\in\R^2$, and
 \item [(ii)] There exists a sequence $\{y_n\}\subset \R^2$ with $\ab{y_n}\to\infty$,  such that  $H(\phi)(y_n)\to0$, as $n\to\infty$.
\end{itemize}
Then any extremizing sequence for inequality \eqref{2-convolution-2} concentrates at infinity.
\end{cor}

\noindent An example of a function that satisfies the hypotheses of the preceding corollary is $\phi(y_1,y_2)=e^{y_1}+e^{y_2}$, $(y_1,y_2)\in\R^2$. The next result shows that extremizing sequences will not concentrate at spatial
infinity if a suitable nondegeneracy condition is placed on the function $\phi$.

\begin{cor}
 \label{StayingBounded}
  Let $\phi:\R^2\to\R$ be a nonnegative, twice continuously differentiable, strictly convex function, such that the set $E:=\{y\in\R^2|\,H(\phi)(y)=0\}$ is 
nonempty. Suppose that there exist $r_0\geq0$ and a function $\Theta:\R^2\to[0,\infty)$ satisfying 
$\inf_{\{\ab{y}>r\}}\Theta(y)>0$, for every $r>r_0$, and such that the matrix
\begin{equation}\label{coercivity_v3}
H(\phi)(y)-\Theta(y)I
\end{equation}
is positive semidefinite, for every $y\in\R^2$. Then every extremizing sequence $\{f_n\}\subset L^2(\R^2)$ for inequality \eqref{2-convolution-2} concentrates along a sequence of points in $E$. Moreover, given any subsequence of $\{f_n\}$, there exist a point $y_0\in E\cap \bar B_{r_0}$ and a sub-subsequence which concentrates at $y_0$.
\end{cor}

\noindent
Condition \eqref{coercivity_v3} implies the existence of a constant $\delta>0$, for which  
$$(\sigma\ast\sigma)(y+z,\psi(y)+\psi(z))\leq(1-\delta) \norma{\sigma\ast\sigma}_{L^\infty},$$ for every $(y,z)\in 
(\R^2)^2\setminus(B_{3r_0}\times B_{3r_0})$ such that $\langle y,z\rangle\geq0$, and Lemma \ref{improvement-implies-concentration} can then be invoked to 
preclude concentration at infinity.  Further note that \eqref{coercivity_v3} is fulfilled by the functions $\phi=|\cdot|^p$, for each $p>2$. In this 
case, we can take 
$r_0=0$, and so concentration can only occur at the origin.
\vspace{.1cm}

\noindent In the case of extremizing sequences concentrating at infinity, we can further refine the analysis as follows. Let $\{f_n\}\subset L^2(\R^2)$ be a sequence such that 
$\norma{f_n}_{L^2}\to 1$, as $n\to\infty$. We say that the sequence $\{f_n\}$ satisfies the \textit{splitting condition} if the following holds. There exists 
$\alpha\in(0,1)$ such that, 
for every $\eps>0$, there exist $r>0$, $n_0\geq 1$, and sequences $\{y_n\}\subset \R^2$, $\{r_n\}\subset \R$, with $r_n\to\infty$, as $n\to\infty$, 
such that the functions $g_{n,1}:=f_{n}\one_{B_r(y_n)}$ and $g_{n,2}:=f_{n}\one_{\R^2\setminus B_{r_n}(y_n)}$ satisfy 
\begin{equation}\label{LionsSplitting}
 \norma{f_{n}-(g_{n,1}+g_{n,2})}_{L^2}^2\leq \eps,\, \ab{\norma{g_{n,1}}_{L^2}^2-\alpha}\leq
\eps,\,\ab{\norma{g_{n,2}}_{L^2}^2-(1-\alpha)}\leq\eps,
\end{equation}
for every $n\geq n_0$. The following result holds.
\begin{prop}
 \label{NoLSplitting}
 Under the hypotheses of Theorem \ref{main}, let $\{f_n\}\subset L^2(\R^2)$ be any extremizing sequence for inequality \eqref{2-convolution-2}. Then 
$\{f_n\}$ does not satisfy the splitting condition.
\end{prop}
\begin{proof}[Sketch of proof]
 Aiming at a contradiction, suppose that the extremizing sequence $\{f_n\}$ satisfies the splitting condition for a given $\alpha\in(0,1)$. 
Let $\eps>0$, and suppose that there exist
$r>0$, $\{y_n\}\subset\R^2$, and $\{r_n\}\subset \R$, for which condition \eqref{LionsSplitting} holds.
Decompose $f_{n}=g_{n,1}+g_{n,2}+h_n$, where $\norma{h_n}_{L^2}^2\leq \eps$.  Part (c) of Lemma \ref{WeakInteraction} implies the uniform estimate
\begin{align*}
 \norma{g_{n,1}\sigma\ast g_{n,2}\sigma}_{L^2}^2
&\leq\rho(r,r_n)^2\norma{g_{n,1}}_{L^2}^2\norma{g_{n,2}}_{L^2}^2,
\end{align*}
where the function
$$\rho(r,r_n):=\sup_{y\in\R^2}\norma{\one_{B_r(y)}\sigma\ast\one_{\R^2\setminus 
B_{r_n}(y)}\sigma}_{L^\infty}^{\frac12}$$
 satisfies $\rho(r,r_n)\to0$, as $n\to\infty$. By an argument similar to the one following \eqref{InterGoToZero} in the proof of Lemma \ref{NoSplitting_part2},  we obtain  
\[\limsup_{n\to\infty}\norma{g_{n,1}}_{L^2}^2\norma{g_{n,2}}_{L^2}^2\leq C\limsup_{n\to\infty}\rho(r,r_n)+C\eps^{\frac 12},	\]
for a universal constant $C<\infty$. We conclude that
\[(1-\alpha-{\eps})(\alpha-{\eps})\leq 
C\eps^{\frac 12},\]
which yields the desired contradiction if $\eps$ is chosen 
small enough, \mbox{depending on $\alpha$.}
\end{proof}

\noindent We finish this chapter by reformulating some of our conclusions in the language of the original concentration-compactness principle of Lions, according to which 
three scenarios may occur: (I) {\it compactness}, (II) {\it vanishing}, or (III) {\it dichotomy}. See \cite[Lemma I.1]{Li} for the precise 
definitions. Up to extraction of subsequences, an extremizing sequence for inequality \eqref{2-convolution-2} which satisfies condition (I) with 
respect to a bounded sequence will concentrate at a point. An extremizing sequence  which satisfies condition (II), or condition (I) with respect to 
an unbounded sequence, will concentrate at infinity. Condition (III) is only possible if neither of the  supports of the split sequence remains in a 
bounded region, in which case the extremizing sequence again concentrates at infinity. Furthermore, if condition (III) 
occurs, then condition (II) must also occur. In this case,  no fixed positive fraction of the $L^2$ mass of an extremizing 
sequence $\{f_n\}$ can remain on any ball of fixed radius, in the limit as $n\to\infty$. To see this, note that the 
proof of \cite[Lemma I.1]{Li}
implies that condition (III)
could otherwise be upgraded to the splitting condition considered above, which in light of Proposition \ref{NoLSplitting} does not hold for any extremizing sequence of \mbox{inequality  
\eqref{2-convolution-2}}.

\section{Sharp Strichartz inequalities}\label{sec:Strichartz}

In this chapter, we consider a number of sharp instances of the Strichartz inequalities \eqref{Strichartz}. 
All cases will follow a common pattern which we now illustrate by focusing on a particular example.
With this purpose in mind, let $\mu=1$ and consider a function $\phi$ as in the statement of Theorem \ref{main}. In this case,  inequality \eqref{Strichartz} can be restated as
\begin{equation}\label{IllustrateStrichartz}
\norma{\mathcal{F}(f(1+\ab{\cdot}^2)^{\frac14}\sigma_\phi)}_{L^4(\R^3)}\lesssim\norma{f}_{L^2(\R^2)}, 
\end{equation}
where the projection measure $\sigma=\sigma_\phi$ is defined in \eqref{defprojmeas}, and the subscript emphasizes that we are no longer taking $\phi=|\cdot|^4$ as in \eqref{quartic0_Intro}.
Inequality \eqref{IllustrateStrichartz} can be rewritten in sharp convolution form as
\begin{equation*}
\norma{f\sqrt{w}\sigma_\phi\ast f\sqrt{w}\sigma_\phi}_{L^2(\R^3)}
\leq 
\mathcal{S}_\phi^2\norma{f}_{L^2(\R^2)}^2,
\end{equation*}
with weight $w=(1+\ab{\cdot}^2)^{\frac12}$ and optimal constant $\mathcal{S}_\phi$. 
The usual Cauchy--Schwarz argument implies
\begin{equation}\label{weightedEstimate}
\norma{f\sqrt{w}\sigma_\phi\ast f\sqrt{w}\sigma_\phi}^2_{L^2(\R^3)}
\leq
\norma{w\sigma_\phi\ast w\sigma_\phi}_{L^\infty(\R^3)}\norma{f}_{L^2(\R^2)}^4,
\end{equation}
whence the upper bound
\begin{equation}
\label{upper-bounds-bc}
\mathcal S_\phi^4\leq \norma{w\sigma_\phi\ast w\sigma_\phi}_{L^\infty(\R^3)}.
\end{equation}
On the other hand, recall formulae \eqref{bdryvaluesDet} and \eqref{ConvFandG},
the boundary values of the  convolution measure $w\sigma_\phi\ast w\sigma_\phi$ are given by
\begin{equation}
\label{bdryValuesWeight}
\bigl(w\sigma_\phi\ast w\sigma_\phi\bigr)(\xi,2\psi(\xi/2))=\frac{\pi 
	w^2(\xi/2)}{\sqrt{\det(H(\psi)(\xi/2))}},
\end{equation}
where we set $\psi=|\cdot|^2+\phi$ as usual.
A slight modification of Lemma \ref{lemma-concentration-point} then yields the lower bound
\begin{equation}
\label{lower-bounds-bc}
\mathcal S_\phi^4\geq \sup_{\xi\in\R^2}\frac{\pi w^2(\xi)}{\sqrt{\det(H(\psi)(\xi))}}.
\end{equation}
Inequalities \eqref{upper-bounds-bc} and \eqref{lower-bounds-bc} provide upper and lower bounds for the value of the optimal constant $\mathcal S_\phi$.
If these bounds happen to coincide, then this determines the value of $\mathcal S_\phi$.
In this case, if the supremum in \eqref{upper-bounds-bc} is achieved {only} at the boundary of the support of the convolution measure, then extremizers are seen not to exist as before. 
In other cases, the following result will be useful in revealing some instances in which  inequality \eqref{lower-bounds-bc} may be strict. 

\begin{lemma}
	\label{lowerBoundExp}
	Given a strictly convex function $\Psi:\R^2\to\R$, consider the measure $\nu(y,t)=\ddirac{t-\Psi(y)}\d y\d t$. 
	Let $E$ denote the support of the convolution measure $\nu\ast\nu$.
	Given $s>0$ and a nonnegative function $w$ on $\R^2$, let $f_s(y)=e^{-s\Psi(y)}\sqrt{w(y)}$. 
	Then the following inequality holds, for every $f_s\in L^2(\R^2)$ for which $f_s\sqrt{w}\nu\ast f_s\sqrt{w}\nu\in L^2(\R^3)$:
	\begin{equation}
	\label{lowerBoundFunctional}
	\frac{\norma{f_s\sqrt{w}\nu\ast 
			f_s\sqrt{w}\nu}_{L^2(\R^3)}^2}{\norma{f_s}_{L^2(\R^2)}^4}
	\geq\frac{\norma{f_s}_{L^2(\R^2)}^4}{\int_E e^{-2s\tau} \d\xi \d\tau}.
	\end{equation}
	 In particular,
	\[ 	\sup_{0\neq f\in L^2(\R^2)}\frac{\norma{f\sqrt{w}\nu\ast 
			f\sqrt{w}\nu}_{L^2(\R^3)}^2}{\norma{f}_{L^2(\R^2)}^4}
	\geq\sup_{s>0}\frac{\norma{f_s}_{L^2(\R^2)}^4}{\int_E e^{-2s\tau} \d\xi \d\tau}.\]
\end{lemma}
\begin{proof}
	For simplicity set $s=1$, the general case being similar. Note that the function
	$f(y)=e^{-\Psi(y)}\sqrt{w(y)}$ coincides with 
	$e^{-t}\sqrt{w(y)}$ 
	on the support of the measure $\nu$. 
	Therefore, the following identities hold:
	\begin{gather*}
	(f\sqrt{w}\nu\ast f\sqrt{w}\nu)(\xi,\tau)=e^{-\tau}(w\nu\ast w\nu)(\xi,\tau),\\
	(f^2\nu\ast f^2\nu)(\xi,\tau)=e^{-\tau}(f\sqrt{w}\nu\ast f\sqrt{w}\nu)(\xi,\tau).
	\end{gather*}
Together with
	\[ \int_{\R^3}(f^2\nu\ast f^2\nu)(\xi,\tau)\d\xi \d\tau
	=\norma{f}_{L^2}^4,\]
	the preceding identities and the Cauchy--Schwarz inequality then imply
	\begin{align*}
	\norma{f}_{L^2}^4
	&=\int_{\R^3} e^{-\tau}(f\sqrt{w}\nu\ast f\sqrt{w}\nu)(\xi,\tau) \d\xi \d\tau\\
	&\leq\Bigl(\int_{E} e^{-2\tau} \d\xi \d\tau\Bigr)^{\frac12}\norma{f\sqrt{w}\nu\ast
		f\sqrt{w}\nu}_{L^2},
	\end{align*}
	from which \eqref{lowerBoundFunctional} easily follows. 
	This completes the proof of the lemma.
\end{proof}

\subsection{Quartic perturbations}\label{sec:QuarticPert} 
We consider a slight generalization of inequality \eqref{quartic0_Intro}, given for $a\geq 0$  by
\[ \norma{\mathcal{F}(f(1+a\ab{\cdot}^2)^{\frac14}\sigma)}_{L^4(\R^3)}
\lesssim\norma{f}_{L^2(\R^2)}, \]
where the measure $\sigma$ is again given by $\sigma(y,t)=\ddirac{t-\ab{y}^2-\ab{y}^4} \d y\d t$.
This inequality can be equivalently rewritten in sharp form as
\begin{equation}
\label{quartic-ss-a}
\norma{f\sqrt{w_a}\sigma\ast f\sqrt{w_a}\sigma}_{L^2(\R^3)}
\leq 
\mathcal{S}_{a}^2\norma{f}_{L^2(\R^2)}^2,
\end{equation}
with weight $w_a=(1+a\ab{\cdot}^2)^{\frac12}$ and optimal constant $\mathcal{S}_{a}$. 
With the notation just introduced, we have the following result, which specialized to $a=1$ yields \mbox{Theorem \ref{bestConstantQuarticIntro}.}

\begin{teo}
	\label{bestConstantQuartic}
	If $0\leq a\leq 2$, then the value of the optimal constant for inequality \eqref{quartic-ss-a} is given by
 $\mathcal{S}^4_{a}=\frac{\pi}{2}.$ 
	Moreover, extremizers for inequality \eqref{quartic-ss-a} do not exist, and extremizing sequences concentrate
	at the origin. If $a>2$, then the following estimates hold:
	\begin{equation}
	\label{boundsAgeq2}
	\max\Bigl\{\frac{\pi}{2},
	\frac{a\sqrt{2\pi}}{8}\Gamma\Bigl(\frac{3}{4}\Bigr)^2\Bigr\}\leq
	\mathcal{S}^4_{a}\leq 
	\frac{a\pi}{4}. 
	\end{equation}
\end{teo}
\begin{proof}
For every $a\geq 0$, the trivial estimate
\[ \norma{f\sigma\ast f\sigma}_{L^2}\leq \norma{\ab{f}\sqrt{w_a}\sigma\ast \ab{f}\sqrt{w_a}\sigma}_{L^2} \]
and Theorem \ref{main} together imply  that $\mathcal{S}^4_{a}\geq \frac{\pi}2$.
This lower bound coincides with the value of the right-hand side of \eqref{lower-bounds-bc} in the special case when $\phi=|\cdot|^4$.
 It follows that
\begin{equation}
\label{lowerupperBoundR}
\frac{\pi}{2}\leq \mathcal{S}^4_{a}\leq \norma{w_a\sigma\ast 
w_a\sigma}_{L^\infty},
\end{equation}
for every $ a\geq 0$.
We are thus reduced to  studying the convolution measure $w_a\sigma\ast w_a\sigma$. 
Formulae \eqref{preHessian} and \eqref{ConvFandG} imply 
\begin{equation}\label{weightedQuartic}
(w_a\sigma\ast w_a\sigma)(\xi,\tau)
=
\int_{\mathbb{S}^1}
\frac{w_a(\xi/2+\alpha\omega)w_a(\xi/2-\alpha\omega)}{
\Big\langle \omega, \frac{\nabla \psi(\xi/2+\alpha\omega)-\nabla \psi(\xi/2-\alpha\omega)}{\alpha}\Big\rangle}
\d\mu_\omega,
\end{equation}
where $\psi=|\cdot|^2+|\cdot|^4$, and the function $\alpha=\alpha(\xi,\tau,\omega)$ is given by \eqref{alphaDef}. 
A straightforward computation shows that the numerator of the integrand in \eqref{weightedQuartic} equals
\begin{equation}\label{numerator}
w_a(\xi/2+\alpha\omega)w_a(\xi/2-\alpha\omega)
=\bigl((1+a(\ab{\xi/2}^2+\alpha^2))^2-a^2\alpha^2\langle\xi,\omega\rangle^2\bigr)^{\frac12}, 
\end{equation}
while the denominator equals
\begin{equation}\label{denominator}
\Big\langle \omega, \frac{\nabla \psi(\xi/2+\alpha\omega)-\nabla 
\psi(\xi/2-\alpha\omega)}{\alpha}\Big\rangle 
=4\bigl(1+2(\ab{\xi/2}^2+\alpha^2)+\langle\xi,\omega\rangle^2\bigr).
\end{equation}
We split the analysis in two cases.

{\bf Case 1:}
$0\leq a\leq 2$. 
To compare \eqref{numerator} and \eqref{denominator}, note that the inequality
\begin{equation}\label{basicineq}
\Bigl(1+a\bigl(\ab{\xi/2}^2+\alpha^2\bigr)\Bigr)^2-a^2\alpha^2\langle\xi,\omega\rangle^2\leq 
\Bigl(1+2\bigl(\ab{\xi/2}^2+\alpha^2\bigr)+\langle\xi,\omega\rangle^2\Bigr)^2
\end{equation}
 holds for every $a\in[0,2],\,\xi\in\R^2,\,\omega\in \mathbb S^1$ and $\alpha\geq 0$.
Moreover, necessary and sufficient conditions for
equality in \eqref{basicineq} to hold for every $\omega\in\mathbb S^1$ are $\xi=0$ when $a=2$, and $\xi=0$ and $\alpha=0$
when $a<2$.
It follows that $\frac14$ is an upper bound for the integrand in \eqref{weightedQuartic}.
Therefore, for every $(\xi,\tau)\in\R^{2+1}$, 
\begin{equation}
\label{upperBoundConv-a}
(w_a\sigma\ast w_a\sigma)(\xi,\tau)\leq \frac{\pi}{2}.
\end{equation}
Moreover, this inequality turns into an equality if and only if $(\xi,\tau)=(0,0)$ when $a<2$, 
and if and only if  $\xi=0$ when $a=2$. To justify this, note that
\[ (w_a\sigma\ast w_a\sigma)(0,\tau)
=\int_{\R^2}\ddirac{\tau-2(\ab{y}^2+\ab{y}^4)}w_a^2(y)\d y
=\frac{\pi}{2}\Bigl(\frac{a}{2}+\frac{1-\frac{a}{2}}{\sqrt{2\tau+1}}\Bigr)\one_{\{\tau\geq 
0\}}(\tau), \]
which specializes to
\[ (w_2\sigma\ast w_2\sigma)(0,\tau)= \frac{\pi}{2}\one_{\{\tau\geq 0\}}(\tau).\]
As a consequence of estimates \eqref{lowerupperBoundR} and \eqref{upperBoundConv-a}, we conclude that $\mathcal{S}^4_{a}= \frac{\pi}{2}$, for every $0\leq a\leq 2$.
Nonexistence of extremizers is a consequence of inequality \eqref{upperBoundConv-a} being strict at almost every point, as in the proof of Theorem \ref{main}. 
Concentration at the origin can likewise be established in an analogous manner.
We point out that the normalized sequence 
$\{f_n/\norma{f_n}_{L^2}\}$, where $f_n(y)=\exp(-n(\ab{y}^2+\ab{y}^4))$, is extremizing  for inequality \eqref{quartic-ss-a}. 

{\bf Case 2:}  $a>2$. Start by noting that the inequality 
\[{\Bigl(1+a\bigl(\ab{{\xi}/{2}}^2+\alpha^2\bigr)\Bigr)^2-a^2\alpha^2\langle\xi,\omega\rangle^2}
\leq 
\frac{a^2}{4} {\Bigl(1+2\bigl(\ab{{\xi}/{2}}^2+\alpha^2\bigr)+\langle\xi,\omega\rangle^2\Bigr)^2}\]
holds for every $a>2,\,\xi\in\R^2,\,\omega\in \mathbb S^1$ and $\alpha\geq 0$. It follows that
\[ (w_a\sigma\ast w_a\sigma)(\xi,\tau)\leq \frac{a\pi}{4}, \]
yielding the upper bound $\mathcal S^4_{a}\leq \frac{a\pi}4$. On the other hand, along the boundary of the support of $w_a\sigma\ast w_a\sigma$,
we have that
\begin{multline*}
(w_a\sigma\ast w_a\sigma)(\xi,2\psi(\xi/2))
=\frac{\pi w^2_a(\xi/2)}{\sqrt{\det(H(\psi)(\xi/2))}}
=\frac{\pi(1+a\ab{\xi/2}^2)}{2\sqrt{(1+2\ab{\xi/2}^2)(1+6\ab{\xi/2}^2)}},
\end{multline*}
and therefore
\[ \mathcal S_{a}^4\geq\sup_{r\geq0} 
\frac{\pi(1+ar)}{2\sqrt{(1+2r)(1+6r)}}\geq \max\Bigl\{\frac{\pi}{2},
\frac{a\pi}{4\sqrt{3}}\Bigr\}. \]
This yields the preliminary bounds
\begin{equation}
\label{prelimBoundRa}
\max\Bigl\{\frac{\pi}{2},
\frac{a\pi}{4\sqrt{3}}\Bigr\}\leq
\mathcal{S}^4_{a}\leq 
\frac{a\pi}{4}. 
\end{equation}
 The lower bound can be sharpened by invoking Lemma \ref{lowerBoundExp}. 
With this purpose in mind, let $f_s(y)=e^{-s\psi(y)}\sqrt{w_a(y)}$.
Its $L^2$ norm is given by
\begin{align*}
\norma{f_s}_{L^2}^2=\int_{\R^2}e^{-2s(\ab{y}^2+\ab{y}^4)}(1+a\ab{y}^2)^{\frac12}\d y
=\pi\int_0^\infty e^{-2s(r+r^2)}(1+ar)^{\frac12}\d r.
\end{align*}
On the other hand, letting $E$ denote the support of the measure $\sigma\ast\sigma$,
\begin{align*}
\int_E e^{-2s\tau}\d\xi \d\tau
=\int_{\R^2}\Big(\int_{2(\ab{\frac{\xi}2}^2+\ab{\frac{\xi}2}^4)}^{\infty} e^{-2s\tau}\d\tau\Big) \d\xi
=\frac{2\pi}{s}\int_0^\infty e^{-4s(r+r^2)}\d r.
\end{align*}
It follows that
\[ \mathcal S_{a}^4\geq \sup_{s>0}\frac{\pi s}{2}\frac{\bigl(\int_0^\infty 
e^{-2s(r+r^2)}(1+ar)^{\frac12}\d r\bigr)^2}{\int_0^\infty e^{-4s(r+r^2)}\d r}. \]
The limit as $s\to 0^+$
of the expression inside this supremum is easily calculated via a change of variables $u=\sqrt{s}r$, yielding
\[ \mathcal S_{a}^4\geq \frac{a\pi}{2}\frac{\bigl(\int_0^\infty 
	e^{-2u^2}u^{\frac 12}\d u\bigr)^2}{\int_0^\infty 
	e^{-4u^2}\d u}=\frac{a\sqrt{2\pi}}{8}\Gamma\Bigl(\frac{3}{4}\Bigr)^2.  \]
\noindent Since $\frac{\sqrt{2\pi}}{8}\Gamma\bigl(\frac{3}{4}\bigr)^2>\frac{\pi}{4\sqrt{3}}$, this indeed sharpens the lower bound in 
\eqref{prelimBoundRa}, and the proof is complete.
\end{proof}

\begin{remark}\label{remark64}
We can consider more general perturbations $\Psi=\ab{\cdot}^2+\ab{\cdot}^4+\phi$, 
with $\phi$ as in the statement of Theorem \ref{main}, satisfying $H(\phi)(0)=0$. 
These correspond to perturbations of the cases considered in Theorem \ref{bestConstantQuartic}. 
Letting $\sigma_\Psi(y,t)=\ddirac{t-\Psi(y)}\d y\d t$,
a similar analysis reveals that, for every $a\in[0,2]$, the sharp inequality
\begin{equation}
\label{quartic_ss_a2}
\norma{f\sqrt{w_a}\sigma_\Psi\ast f\sqrt{w_a}\sigma_\Psi}^2_{L^2(\R^3)}
\leq 
\frac{\pi}2\norma{f}_{L^2(\R^2)}^4
\end{equation}
holds,  extremizers do not exist, and  extremizing sequences concentrate at the origin. 
\end{remark}

\subsection{Convolutions of pure powers} \label{sec:ConvPP}
In this section, we study the convolution of the projection measure 
\begin{equation}\label{defnup}
\nu_p(y,t)=\ddirac{t-\ab{y}^p} \d y \d t,
\end{equation}
where $p\geq 2$ and $(y,t)\in\R^{2+1}$. A scaling 
argument shows that there exists a unique possible Strichartz estimate  in $L^4(\R^3)$, namely
\begin{equation}\label{onepossiblepureStrichartz}
\norma{\mathcal F(f\ab{\cdot}^{\frac{p-2}4}\nu_p)}_{L^4(\R^3)}\lesssim\norma{f}_{L^2(\R^2)}.
\end{equation}
\noindent As before, the analysis of the sharp form of inequality \eqref{onepossiblepureStrichartz} leads to the study of the  convolution measure $w\nu_p\ast w\nu_p$, 
with weight $w=\ab{\cdot}^{\frac{p-2}2}$. We record its main properties in the following result, which should be compared to Proposition \ref{ConvolFormulaProp}.

\begin{prop}
	\label{convPowers}
	Given $p\geq 2$, let $w=\ab{\cdot}^{\frac{p-2}2}$. 
	Let $\nu_p$ be the measure defined by \eqref{defnup}.
	Then the following assertions hold for the convolution measure $w\nu_p\ast w\nu_p$:
	\begin{itemize}
		\item[(a)] It is absolutely continuous with respect to Lebesgue measure on $\R^3$.
		\item[(b)] Its support, denoted $E_p$, is given by
		\[ E_p=\{(\xi,\tau)\in\R^{2+1}:\tau\geq {2^{1-p}}{\ab{\xi}^p}\}. \]
		\item[(c)] Its Radon--Nikodym derivative, also denoted by $w\nu_p\ast w\nu_p$, defines a bounded 
		continuous function in the interior of the set $E_p$.
		\item[(d)] It is radial in $\xi$, and homogeneous of
		degree zero in the sense that
		$$(w\nu_p\ast w\nu_p)(\la\xi,\la^p\tau)=(w\nu_p\ast w\nu_p)(\xi,\tau), 
		\text{ for every } \la>0.$$
		\item[(e)] It extends continuously to the boundary of $E_p$, except at the point $(\xi,\tau)=(0,0)$, with values given by
		\[(w\nu_p\ast w\nu_p)(\xi,{{2^{1-p}}\ab{\xi}^p})=\frac{\pi}{p\sqrt{p-1}}, \text{ if } \xi\neq 
		0.\]
		\item[(f)] If $p>2$, then the maximum value of $w\nu_p\ast w\nu_p$ is only attained along the vertical axis	$\{(0,\tau): \tau> 0\}$, where it equals $\frac{\pi}p$.
	\end{itemize}
\end{prop}
\begin{proof}
Properties (a) and (b) follow as in the proof of Proposition \ref{ConvolFormulaProp}. 
Property (d) is also straightforward to check. 
We then start by showing that $w\nu_p\ast w\nu_p$ defines a continuous function inside its support. 
Reasoning as in \eqref{DeltaCalc}, we have that
$$(w\nu_p\ast w\nu_p)(\xi,\tau)
=\int_{\R^2}\ddirac{\tau-\ab{\tfrac{\xi}{2}+y}^p-\ab{\tfrac{\xi}{2}-y}^p}
\ab{\tfrac{\xi}{2}+y}^{\frac{p-2}2}\ab{\tfrac{\xi}{2}-y}^{\frac{p-2}2} \d y.$$
Write $\tau=\la\ab{\xi}^p$ with $\xi=2e_1$, 
where $e_1$ denotes the first canonical vector. Changing to polar coordinates with polar axis parallel to $e_1$, we obtain
$$(w\nu_p\ast w\nu_p)(\xi,\la\ab{\xi}^p)=
\int_0^{2\pi}\int_0^\infty\ddirac{2^p\la-2-\vphi_{\te}(r)}
((r^2+1)^2-4r^2\cos^2\te)^{\frac{p-2}4}r \d r\d\te,$$
where the function 
\begin{equation}
\label{implicitVphi}
\vphi_{\te}(r):=(r^2+1+2r\cos\te)^{\frac p2}+(r^2+1-2r\cos\te)^{\frac p2}-2
\end{equation}
is convex in the 
variable $r$ for each fixed $\te$, with unique global minimum at $r=0$ as a result of Lemma 
\ref{convexity}. A change of variables $s=\vphi_{\te}(r)$ yields
\begin{align}
(w\nu_p\ast w\nu_p)(\xi,\la\ab{\xi}^p)&=
\int_0^{2\pi}\int_0^\infty\ddirac{2^p\la-2-s}
\frac{((r^2+1)^2-4r^2\cos^2\te)^{\frac{p-2}4}r}{\vphi_\te'(r)}\d s\d\te\nonumber\\
\label{formulaPowerP}
&=\one_{\{\la\geq{2^{1-p}}\}}(\la)\int_0^{2\pi}((r^2+1)^2-4r^2\cos^2\te)^{\frac{p-2}4}
	\Bigl(\frac{r}{\vphi_\te'(r)}\Bigr)\, \d\te,
\end{align}
where $r=\vphi_{\te}^{-1}(2^p\la-2)$, and $\vphi_\te'$ denotes the  
derivative of the function $\vphi_\te$ with respect to $r$. A calculation shows that 
\begin{align}
\vphi_\te'(r)
&=pr\bigl((r^2+1+2r\cos\te)^{\frac{p-2}2}+(r^2+1-2r\cos\te)^{\frac{p-2}2}\bigr)\notag\\
&\quad+p\cos\te\bigl((r^2+1+2r\cos\te)^{\frac{p-2}2}-(r^2+1-2r\cos\te)^{\frac{p-2}2}\bigr),\label{phiprime}
\end{align}
which only vanishes at $r=0$.
Using part (d), we see that $(w\nu_p\ast w\nu_p)(\xi,\tau)=(w\nu_p\ast w\nu_p)(e_1, \ab{\xi}^{-p}{\tau})$, for every $\xi\neq 0$. Therefore continuity of the convolution measure at a point $(\xi,\tau)$  in the interior of the set $E_p$, for $\xi\neq 0$, follows from that of $(w\nu_p\ast w\nu_p)(e_1,\la)$, for $\la>2^{1-p}$. The latter is seen to hold via the Implicit Function Theorem, given that $r$ is a differentiable function of $\la$ and $\te$. As for continuity along the positive $\tau$-axis, note that, given $\tau>0$ and a sequence $(\xi_n,\tau_n)\to(0,\tau)$, as $n\to\infty$, with $\xi_n\neq 0$, for every $n$, we have
 \begin{equation}\label{verticalaxisvalue}
  (w\nu_p\ast w\nu_p)(\xi_n,\tau_n)
  =(w\nu_p\ast w\nu_p)(e_1,\tfrac{\tau_n}{\ab{\xi_n}^p})\to \int_0^{2\pi}\frac{\d\te}{2p}
  =\frac{\pi}{p},\text{ as }n\to\infty. 
  \end{equation}
 Here we used that $\la_n:=\frac{\tau_n}{\ab{\xi_n}^p}\to\infty$, and that $r=r(\la_n,\te)\to\infty$ for each fixed $\te$, as $n\to\infty$.
  Boundedness is a 
consequence of the inequality 
\begin{equation}\label{numdenomineq}
2pr(r^2+1+2r\cos\te)^{\frac{p-2}4}(r^2+1-2r\cos\te)^{\frac{p-2}4}\leq \vphi_\te'(r),
\end{equation}
which holds for every $r\geq 0$, $\te\in[0,2\pi]$ and $p\geq 2$. To verify \eqref{numdenomineq}, recall expression \eqref{phiprime} for $\vphi_\te'$, and note that, as long as $p\geq 2$,
\[ p\cos\te\bigl((r^2+1+2r\cos\te)^{\frac{p-2}2}-(r^2+1-2r\cos\te)^{\frac{p-2}2}\bigr)\geq 0, \]
for every $r\geq 0$ and $\te\in[0,2\pi]$.
This concludes the verification of  (c).
We can continuously extend the value of 
the function $(w\nu_p\ast w\nu_p)(\xi,\la\ab{\xi}^p)$ to $\la=2^{1-p}$ by noting that $r\to 0^+$ as $\la\to(2^{1-p})^+$. This yields the following value for the extension:
\[(w\nu_p\ast w\nu_p)(\xi,{2^{1-p}}\ab{\xi}^p)
=\int_0^{2\pi}\frac{\d\te}{\vphi_\te''(0)}
=\int_0^{2\pi}\frac{\d\te}{2p(1+(p-2)\cos^2\te)}
=\frac{\pi}{p\sqrt{p-1}}.\]
Note that this coincides with the value predicted by the analogous of formula \eqref{bdryValuesWeight}.
Property (e) is now proved.
Finally, if $p>2$, then  a discussion of the cases of equality in  \eqref{numdenomineq} reveals that the strict inequality
\begin{equation}\label{StrictOutsideVerticalAxis}
(w\nu_p\ast w\nu_p)(\xi,\tau)<
\frac{\pi}{p}\one_{\{\tau\geq{2^{1-p}}\ab{\xi}^p\}}(\xi,\tau)
\end{equation}
holds for every $(\xi,\tau)$ with $\xi\neq 0$.
Moreover, the value along the $\tau$-axis was already calculated in \eqref{verticalaxisvalue}, is alternatively given by
\[ (w\nu_p\ast w\nu_p)(0,\tau)
=\int_{\R^2}\ddirac{\tau-2\ab{y}^p}\ab{y}^{p-2}\d y=\frac{\pi}{p}\one_{\{\tau> 0\}}(\tau), \]
and therefore equals the maximum value. 
This concludes the verification of (f) and the proof of the proposition.
\end{proof}

\begin{remark}\label{boundsQp}
	The boundedness of $w\nu_p\ast w\nu_p$ given by part (c) of Proposition \ref{convPowers} implies the validity of  the  Strichartz estimate \eqref{onepossiblepureStrichartz}.
	Moreover, parts (e) and (f) imply that the optimal constant $\mathcal Q_p$ for the corresponding sharp inequality in convolution form,
	 	\begin{equation} \label{onepossiblepureStrichartz4}
	 \norma{f\sqrt{w}\nu_p\ast f\sqrt{w}\nu_p}_{L^2(\R^3)}\leq \mathcal Q_p^2\norma{f}_{L^2(\R^2)}^2, 
	 	\end{equation}
	satisfies
	\begin{equation}\label{lowerupperboundforQ}
	 \frac{\pi}{p\sqrt{p-1}}\leq \mathcal Q^4_p\leq \frac{\pi}{p}.
	 \end{equation}
	 	Contrary to the case of the quartic perturbation studied in \S
	\ref{sec:QuarticPert}, this does not determine 
	$\mathcal Q_p$ since the 
	upper and lower bounds do not coincide for $p>2$.
\end{remark}

\noindent  In order to sharpen the lower bound in \eqref{lowerupperboundforQ}, we will use the following straightforward consequence of Lemma \ref{lowerBoundExp}.
\begin{cor}
	\label{lowerBoundPurePowers}
	Given $p\geq 2$, let $\nu_p(y,t)=\ddirac{t-\ab{y}^p} \d y\d t$, and $w=\ab{\cdot}^{\frac{p-2}2}$. 	
	Then the following estimate holds for the function $f(y)=\exp(-\ab{y}^p)\ab{y}^{\frac{p-2}4}$:
	\begin{equation}
	\label{boundGammaPP}
	\frac{\norma{f\sqrt{w}\nu_p\ast 
			f\sqrt{w}\nu_p}_{L^2(\R^3)}^2}{\norma{f}_{L^2(\R^2)}^4}
	\geq\frac{\pi}{p2^{1-\frac{2}{p}}}\frac{\Gamma\bigl(\frac{1}{2}+\frac{1}{p}\bigr)^2}
	{\Gamma\bigl(\frac{2}{p}\bigr)}.
	\end{equation}
\end{cor}

\subsection{The pure quartic}\label{sec:PureQuartic}
In this section, we consider the  case $p=4$ of \eqref{onepossiblepureStrichartz4}. 
Let $\nu=\nu_4$ be given by \eqref{defnup}. 
The next result records the additional simplifications which appear in the integral formula \eqref{formulaPowerP} for the convolution $|\cdot|\nu\ast |\cdot|\nu$.

\begin{cor}
	\label{formulaConvQuartic}
	Let $\nu(y,t)=\ddirac{t-\ab{y}^4} \d y \d t$. Then the following 
	integral formula holds, for every $\xi\neq 0$ and $\la\geq \frac18$:
	\begin{multline}
	\label{formulaPureQuartic}
	(|\cdot|\nu\ast |\cdot|\nu)(\xi,\la\ab{\xi}^4)\\
	=\frac{1}{4\sqrt{2}}\int_0^{2\pi}\biggl(\frac{\la+\cos^2\te+2\cos^4\te-2
		(2\la+\cos^2\te+\cos^4\te)^{\frac12}\cos^2\te}{2\la+\cos^2\te+\cos^4\te}\biggr)^{\frac{1}{2}}\d\te.
	\end{multline}
	Additionally,
	$$ (|\cdot|\nu\ast |\cdot|\nu)(0,\tau)=\frac{\pi}{4}\one_{\{\tau>0\}}(\tau),
\text{ and }
 (|\cdot|\nu\ast |\cdot|\nu)(\xi,\tfrac{\ab{\xi}^4}{8})=\frac{\pi}{4\sqrt{3}},\text{ if } \xi\neq 0.$$
\end{cor}

\noindent We are now ready to prove Theorem \ref{extPureQuarticIntro}.

\begin{proof}[Proof of Theorem \ref{extPureQuarticIntro}]
	In view of \cite[Theorem 4.1]{JSS},  the existence of extremizers for inequality \eqref{hom_conv_Sobolev_Strichartz_Intro}  follows from the strict inequality
	$\mathcal Q^4>\frac{\pi}{4\sqrt{3}}.$
	In order to establish it,  consider the function $f(y)=\exp(-|y|^4){\ab{y}^{\frac 12}}$. Invoking Corollary \ref{lowerBoundPurePowers}, we have that
	\begin{equation}\label{num1}
	\mathcal Q^4
	\geq \frac{\norma{f|\cdot|^{\frac 12}\nu\ast 
			f|\cdot|^{\frac 12}\nu}_{L^2}^2}{\norma{f}_{L^2}^4}
	\geq\frac{\pi}{4\sqrt{2}}\frac{\Gamma\bigl(\frac{3}{4}\bigr)^2}
	{\Gamma\bigl(\frac{1}{2}\bigr)}
	=\frac{\sqrt{2\pi}}{8}\Gamma\Bigl(\frac{3}{4}\Bigr)^2
	>\frac{\pi}{4\sqrt{3}},
	\end{equation}
	as desired. 
	The upper bound $\mathcal Q^4\leq \frac{\pi}4$ holds in view of part (f) of Proposition \ref{convPowers} for $p=4$.
	That this upper bound	
	is strict follows from the existence of extremizers, and the fact that the pointwise inequality $|\cdot|\nu\ast |\cdot|\nu< \frac{\pi}4$ is strict almost everywhere, as quantified by \eqref{StrictOutsideVerticalAxis}.
\end{proof}

\begin{remark}
	A direct calculation shows that the function $f(y)=\exp(-\ab{y}^4){\ab{y}^{\frac 12}}$ satisfies
	\begin{equation}\label{numericalformula4}
	\frac{\norma{f |\cdot|^{\frac 12}\nu\ast 
			f|\cdot|^{\frac 12}\nu}_{L^2(\R^3)}^2}{\norma{f}_{L^2(\R^2)}^4}
	=\frac{2}{\sqrt{\pi}\Gamma(\frac{3}{4})^2}\int_{0}^{2\sqrt{2}}
	(|\cdot|\nu\ast |\cdot|\nu)^2(e_1,{t^{-2}}) \d t.
	\end{equation}
 Invoking formula \eqref{formulaPureQuartic} for the convolution $|\cdot|\nu\ast |\cdot|\nu$, the integral on the right-hand side of \eqref{numericalformula4} can be evaluated numerically. 
 With precision $5\times 10^{-6}$, one checks that
	\begin{equation}\label{NumValue}
	\frac{\norma{f|\cdot|^{\frac 12}\nu\ast f|\cdot|^{\frac 12}\nu}_{L^2(\R^3)}^2}{\norma{f}_{L^2(\R^2)}^4}\approx 
	0.489333. 
	\end{equation}
		Note that the lower bound $\frac{\sqrt{2\pi}}{8}\Gamma(\frac{3}{4})^2\approx 0.470508$ obtained in \eqref{num1} already amounts to about $96\%$ of the value in \eqref{NumValue}. 
		This indicates that the Cauchy--Schwarz argument from Lemma \ref{lowerBoundExp} is quite sharp for $p=4$. We expect the same argument to work for other values of $p$ as well, and remark on that in the next section.\\
\end{remark}

\subsection{Other pure powers}\label{sec:OtherPP}
In this section, we briefly comment on how to approach  the problem of existence of extremizers for inequality  \eqref{onepossiblepureStrichartz4} in the case of pure powers other than the quartic. 
Given  $p>2$,
let $\mathcal Q_p$ be the optimal constant in inequality \eqref{onepossiblepureStrichartz4}.
The following result provides a partial replacement for Theorem \ref{extPureQuarticIntro} \mbox{when $p\neq 4$.}
\begin{prop}
	\label{lowerBoundQp}
	There exists $p_0>5$
	such that, for every $p\in(2,p_0)$, 
	\[ \frac{\pi}{p\sqrt{p-1}}<\mathcal Q_p^4\leq \frac{\pi}{p}. \]
\end{prop}
\begin{proof}[Sketch of proof]
	The upper bound holds in view of part (f) of Proposition \ref{convPowers}.
	Invoking Corollary \ref{lowerBoundPurePowers} as before, we obtain the lower bound \eqref{boundGammaPP}.
	We are thus reduced to showing that
	\[ \frac{\pi}{p2^{1-\frac{2}{p}}}\frac{\Gamma\bigl(\frac{1}{2}+\frac{1}{p}\bigr)^2}
	{\Gamma\bigl(\frac{2}{p}\bigr)}>\frac{\pi}{p\sqrt{p-1}}, \]
	or equivalently
	\begin{equation}
	\label{ineqGammaP}
	\Gamma\Bigl(\frac{1}{2}+\frac{1}{p}\Bigr)^2
	>\frac{2^{1-\frac{2}{p}}}{\sqrt{p-1}}\Gamma\Bigl(\frac{2}{p}\Bigr).
	\end{equation}
	Figure \ref{fig:Gamma} below illustrates the validity of this inequality inside the claimed range.
	\end{proof}
	\begin{figure}[htbp]
  \centering
  \includegraphics[height=5cm]{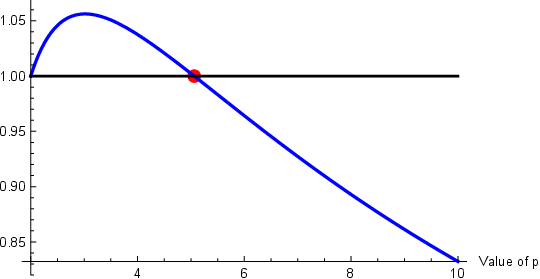} 
    \caption{Plot of the ratio $\frac{\text{LHS}}{\text{RHS}}$ of inequality \eqref{ineqGammaP}, for $2<p<10$. The $p$-coordinate of the intersection (red) point has been numerically determined and equals $5.061147$ (6 d.p.).}
\label{fig:Gamma}
\end{figure}
\noindent  Let $p_0$  be the exponent promised by Proposition \ref{lowerBoundQp}. 
For every $p\in (2,p_0)$,  
extremizing sequences for inequality \eqref{onepossiblepureStrichartz4} are seen not to
concentrate at a point of the boundary, except possibly at the origin.
Then, arguments similar to the ones from \cite{CS, JSS}
can presumably establish the 
existence of extremizers, provided that a ``cap bound'' holds, together with a principle quantifying the weak interaction between distant 
caps, in the spirit of Lemma \ref{WeakInteraction}.
 As a final remark, we record the following generalization of formula \eqref{numericalformula4} for generic values of $p>2$,
\begin{align*}
\frac{\norma{f\sqrt{w}\nu_p\ast f\sqrt{w}\nu_p}_{L^2(\R^3)}^2}{\norma{f}_{L^2(\R^2)}^4}
=\frac{p}{2\pi}\frac{\Gamma\bigl(\frac{2}{p}\bigr)}{\Gamma\bigl(\frac{1}{2}+\frac{1}{p}\bigr)^2}
\int_{0}^{2^{2-\frac 2p}} (w\nu_p\ast w\nu_p)^2(e_1,{t^{-\frac p2}})\d t,
\end{align*}
which could be of interest for further numerical explorations.

\section*{Acknowledgments} 
The software {\it Mathematica} and the open software packages \textit{Maxima} and \textit{Scilab} were  used to perform the numerical tasks described in Chapter \ref{sec:Strichartz}. The beginning of this work was accomplished during an extended research visit of the second author to the Hausdorff Institute for Mathematics, whose hospitality is greatly appreciated. We thank Mateus Sousa for  reading a preliminary version of this manuscript, and Stefan Steinerberger and Christoph Thiele for various comments and suggestions. Finally, we are indebted to the anonymous referee for pointing out several related questions in the existing literature.


\begin{thebibliography}{03}

\bibitem{BM} Jong-Guk Bak and David McMichael, 
\newblock {\it Convolution of a measure with itself and a restriction theorem.}
\newblock Proc. Amer. Math. Soc. {\bf 125} (1997), no.~2, 463--470. 

\bibitem{BKS}
Matania Ben-Artzi, Herbert Koch and Jean-Claude Saut, 
\newblock {\it Dispersion estimates for fourth order Schr\"odinger equations.} 
\newblock C. R. Acad. Sci. Paris S�r. I Math. {\bf 330} (2000), no.~2, 87--92. 

\bibitem{BBCH} Jonathan Bennett, Neal Bez, Anthony Carbery and Dirk Hundertmark, 
\newblock {\it Heat-flow monotonicity of Strichartz norms.} 
\newblock Anal. PDE {\bf 2} (2009), no.~2, 147--158. 

\bibitem{BS} Neal Bez and Mitsuru Sugimoto,
\newblock {\it Optimal constants and extremisers for some smoothing estimates.}
\newblock Preprint, 2012. arXiv:1206.5110.

\bibitem{C} Emanuel Carneiro, 
\newblock{\it A sharp inequality for the Strichartz norm.}  
\newblock Int. Math. Res. Not. IMRN 2009, no.~16, 3127--3145. 

\bibitem{COS} Emanuel Carneiro and Diogo Oliveira e Silva, 
\newblock{\it Some sharp restriction inequalities on the sphere.} 
\newblock Int. Math. Res. Not. IMRN 2015, no.~17, 8233--8267. 

\bibitem{CS} Michael Christ and Shuanglin Shao, 
\newblock{\it Existence of extremals for a Fourier restriction inequality.}
\newblock Anal. PDE {\bf 5} (2012), no.~2, 261--312. 

\bibitem{CS2} Michael Christ and Shuanglin Shao, 
\newblock{\it On the extremizers of an adjoint Fourier restriction inequality.}
\newblock Adv. Math. {\bf 230} (2012), no.~3, 957--977. 

\bibitem{FVV} Luca Fanelli, Luis Vega and Nicola Visciglia, 
\newblock{\it On the existence of maximizers for a family of restriction theorems.}  
\newblock Bull. Lond. Math. Soc. {\bf 43} (2011), no.~4, 811--817. 

\bibitem{FVV2} Luca Fanelli, Luis Vega and Nicola Visciglia,
\newblock{\it Existence of maximizers for Sobolev--Strichartz inequalities.}
\newblock Adv. Math. {\bf 229} (2012), no.~3, 1912--1923. 

\bibitem{F} Damiano Foschi,
\newblock {\it Maximizers for the Strichartz inequality.} 
\newblock J. Eur. Math. Soc. (JEMS) {\bf 9} (2007), no.~4, 739--774. 

\bibitem{F2} Damiano Foschi, 
\newblock {\it Global maximizers for the sphere adjoint Fourier restriction inequality.} 
\newblock J. Funct. Anal. {\bf 268} (2015), no.~3, 690--702. 

\bibitem{FOS} Damiano Foschi and Diogo Oliveira e Silva, 
\newblock {\it Some recent progress in sharp Fourier restriction theory.}
Preprint, 2016.

\bibitem{FLS} Rupert Frank, Elliott H. Lieb and Julien Sabin,
\newblock {\it Maximizers for the Stein--Tomas inequality.}
\newblock Preprint, 2016. arXiv:1603.07658.
To appear in Geometric and Functional Analysis.

\bibitem{Han} Wei Han, 
\newblock {\it The sharp Strichartz and Sobolev--Strichartz inequalities for the fourth order 
Schr\"odinger equation.} 
\newblock Math. Meth. Appl. Sci. 2015, {\bf 38}, 1506--1514. 

\bibitem{HUL} Jean-Baptiste Hiriart-Urruty and Claude Lemar\'echal, 
\newblock {\it Fundamentals of convex analysis.} 
\newblock Grundlehren Text Editions. Springer-Verlag, Berlin, 2001.

\bibitem{HZ} Dirk Hundertmark and Vadim Zharnitsky, 
\newblock {\it On sharp Strichartz inequalities in low dimensions.} 
\newblock Int. Math. Res. Not. 2006, Art. ID 34080, 18 pp. 

\bibitem{JPS} Jin-Cheng Jiang, Benoit Pausader and Shuanglin Shao, 
\newblock {\it The linear profile decomposition for the fourth order Schr\"odinger equation.}
\newblock J. Differential Equations {\bf 249} (2010), no.~10, 2521--2547. 

\bibitem{JSS} Jin-Cheng Jiang, Shuanglin Shao and Betsy Stovall,
\newblock {\it Linear profile decompositions for a family of fourth order Schr\"odinger equations.}
\newblock Preprint, 2014. 	arXiv:1410.7520.

\bibitem{KT} Markus Keel and Terence Tao,
\newblock {\it Endpoint Strichartz estimates.}
\newblock Amer. J. Math. {\bf 120} (1998), no. 5, 955--980. 

\bibitem{KPV} Carlos Kenig, Gustavo Ponce and Luis Vega,
\newblock {\it Oscillatory integrals and regularity of dispersive equations.} 
\newblock Indiana Univ. Math. J. {\bf 40} (1991), no.~1, 33--69. 

\bibitem{Li} Pierre-Louis Lions,
\newblock {\it The concentration-compactness principle in the calculus of variations. The locally compact case. I.}
\newblock Ann. Inst. H. Poincar\'e Anal. Non Lin\'eaire, {\bf 1} (1984), no.~2, 109--145.

\bibitem{OS2} Diogo Oliveira e Silva,
\newblock {\it Extremizers for Fourier restriction inequalities: convex arcs.}
\newblock J. Anal. Math. {\bf 124} (2014), 337--385. 

\bibitem{OS} Diogo Oliveira e Silva, 
\newblock {\it Nonexistence of extremizers for certain convex curves.}
\newblock Preprint, 2012. arXiv:1210.0585. To appear in Mathematical Research Letters. 

\bibitem{Pa} Benoit Pausader, 
\newblock {\it Global well-posedness for energy critical fourth-order Schr\"odinger equations in the radial case.}
\newblock Dyn. Partial Differ. Equ. {\bf 4} (2007), no. 3, 197--225. 

\bibitem{Q} Ren\'e Quilodr\'an, 
\newblock {\it Nonexistence of extremals for the adjoint restriction inequality on the hyperboloid.} 
\newblock J. Anal. Math. {\bf 125} (2015), 37--70. 

\bibitem{R} Javier Ramos,
\newblock {\it A refinement of the Strichartz inequality for the wave equation with applications.} 
\newblock Adv. Math. {\bf 230} (2012), no.~2, 649--698. 

\bibitem{RS} Michael Ruzhansky and Mitsuru Sugimoto,
\newblock {\it Smoothing properties of evolution equations via canonical transforms and comparison principle.} 
\newblock Proc. Lond. Math. Soc. (3) {\bf 105} (2012), no.~2, 393--423. 

\bibitem{S} Elias M. Stein,
\newblock {\it Harmonic Analysis: Real-Variable Methods, Orthogonality, and Oscillatory Integrals,} 
\newblock Princeton Univ. Press, Princeton, NJ, 1993.

\bibitem{St} Robert S. Strichartz, 
\newblock {\it Restrictions of Fourier transforms to quadratic surfaces and decay of solutions of wave equations.} 
\newblock Duke Math. J. {\bf 44} (1977), no.~3, 705--714. 

\bibitem{T} Peter A. Tomas, 
\newblock {\it A restriction theorem for the Fourier transform.} 
\newblock Bull. Amer. Math. Soc. {\bf 81} (1975), 477--478. 

\end{thebibliography}
\end{document}